\documentclass[11pt]{article}
\usepackage[margin=1in]{geometry}  
\usepackage[latin1]{inputenc}
\usepackage{amsmath}
\usepackage{amsfonts}
\usepackage{amssymb}
\usepackage{amsthm}
\usepackage{makeidx}
\usepackage{graphicx}
\usepackage{tikz}
\usepackage[square, comma, numbers, sort & compress]{natbib} 
\usepackage{hyperref}
\usepackage{float}
\usepackage[font=small,labelfont=bf]{caption}
\usepackage[font=small,labelfont=bf]{subcaption}
\usepackage{tikz}
\usepackage{resizegather}
\usepackage{empheq}
\usepackage{xcolor}

\newtheorem{thm}{Theorem}
\newtheorem{defn}[thm]{Definition}

\newtheorem{lem}[thm]{Lemma}
\newtheorem{rem}[thm]{Remark}

\newcommand{\R}{\mathrm{I\hspace{-0.5ex}R}}

\newcommand {\dt}{\Delta{t}}
\newcommand {\dx}{\Delta{x}}

\allowdisplaybreaks
\title{The effect of artificial viscosity on numerical boundary feedback control of linear hyperbolic systems  }
\date{\today}
\author{M.K. Banda\footnote{Department of Mathematics and Applied Mathematics, University of Pretoria, South Africa} \and G.Y. Weldegiyorgis\footnotemark[1]}

\begin{document}
\maketitle

\begin{abstract}
A numerical analysis of the effect of artificial viscosity is undertaken in order to understand the effect of numerical diffusion on numerical boundary feedback control. The analysis is undertaken on the linear hyperbolic systems discretised using the upwind scheme. The upwind scheme solves the advection-diffusion equation with up to second-order accuracy. The analysis shows that the upwind scheme with CFL equal to one gives the expected theoretical decay up to first-order. On the other hand the upwind scheme with CFL less than one gives decay depending on the second derivative of the data and the CFL number. Further the decay rates deteriorate if the second derivatives of the solution are small. Thus the decay rates computed by the numerical schemes tend to be higher in comparison to the theoretical prediction. Computations on test cases which include isothermal Euler and the St Venant Equations confirm the analytical results.
 
% {\bf STORYLINE}
% \begin{itemize}
% \item UPWIND WITH CFL EQ 1 GIVES EXACT DECAY UP TO ORDER O(DX)
% \item CFL LQ 1 GIVES DIFFERENT DECAY  (DEPENDS ON 2nd DERIVATIVE OF DATA AND CFL)
% \item CONJECTURE: UPWIND SOLVES DIFFUSION UP TO ORDER O(DX**2)
% \item ANALYSIS SUGGEST WORSE RATE FOR UPWIND (IF 2ND DERIVATIVE IS SMALL)
% \item SUBTRACTING DIFFUSION TO UPWIND YIELDS SAME RATE AGAIN
% \item FIRST OBSERVATIONS: CONSTANT DATA: Really get a result in the lines $ - \alpha \mu$, $-\alpha \mu + \mu^2 \epsilon$
% \item SECOND OBSERVATION: SUBTRACTING THE DIFFUSION GETS THE EXPECTED DECAY FOR NON-CONSTANT DATA
% \end{itemize}


\end{abstract}

\section{Introduction}
For many physical systems described by hyperbolic equations, 
boundary feedback stabilisation based on local observations
is possible, see \cite{Coron_2009,Komornik:1997aa}. Based on 
a novel Lyapunov function,  progress on the analytical treatment of stabilisation 
of hyperbolic equations has been achieved and, in particular, 
for isothermal Euler equations  \cite{M.Gugat:2011aa,M.Dick:2010aa,M.Gugat:2012aa} and shallow--water equations 
\cite{J.-M.Coron:2002aa,J.-M.Coron:2008ab,J.deHalleux:2003aa,M.Gugat:2003aa,G.Leugering:2002aa,J.-M.Coron:2009aa,M.Gugat:2004aa} 
stabilisation results on the continuous setting are available.  Besides linearisation of nonlinear hyperbolic equations \cite{J.-M.Coron:2009aa}, 
feedback stabilisation of linear hyperbolic equations has been considered for a network of vibrating strings \cite{Bastin:2016aa,Gerster:2019aa} or as a model for electric transmission networks \cite{P.Schillen:2017aa}. The most common underlying tool is a suitably weighted $L^2 \text{ or } H^2$ norm used as a
Lyapunov function that gives rise to exponential decay under the so--called {\em dissipative} boundary conditions \cite{J.-M.Coron:2008ab,J.-M.Coron:2007ab}. It must be noted that in this framework, a comparison of the stabilisation concept with other existing results has been presented in \cite{J.-M.Coron:2015aa}. Therein it can be observed that stability in higher order Sobolev norms
also yields stabilisation of nonlinear systems \cite{J.-M.Coron:2007ac,Coron_2009}. 

While the research on  Lyapunov functions and stabilisation results is usually restricted to the analytical formulation only a few results on suitable numerical discretisation schemes and their corresponding analysis exist \cite{Herty:2019aa,P.Schillen:2017aa,M.K.Banda:2013aa, Banda_2018}. The analysis therein also presented explicit decay rates for the considered numerical schemes. In \cite{M.K.Banda:2013aa} exponential decay on a finite time horizon has been established for nonlinear conservations laws and the results have been extended to linear hyperbolic balance laws in~\cite{Herty:2019aa,P.Schillen:2017aa, Banda_2018} also with respect to higher--order norms. 

However, in the numerical simulations it has been observed that the actual decay of the discretised Lyapunov function is  much stronger compared to the theoretical prediction for the discrete scheme \cite{M.K.Banda:2013aa}. It was suggested that the additional decay is related to the presence of numerical diffusion in the numerical scheme. In fact, it is well--known that finite--volume schemes solve a modified (diffusive) equation to higher--order than the corresponding hyperbolic equation, the reader may refer to \cite{LeVeque:2002aa}. The goal of the present work is to clarify this conjecture and to quantify the (possibly small) diffusive effect on the numerical stabilisation result.  As in the previous work, we restrict ourselves in this paper to the case of linear  $2\times 2$ systems with one negative and one positive eigenvalue. However, the main result also remains valid in the case of $n\times n$  linear hyperbolic systems with eigenvalues not equal to zero. 

The paper is organised as follows: we present a motivating computation on the continuous level in Section \ref{sec1} before presenting the main result in Section \ref{sec2}. The latter states explicit decay rates of discretised linear hyperbolic equations perturbed by viscosity. Numerical results in Section \ref{sec3} validate the theoretical findings. 

\section{Motivation and Definition of the Problem} \label{sec1}

We start by briefly recall notation and existing results. Consider the stability of a steady state of a physical system described by a linear $2\times 2$ system of hyperbolic differential equations without source terms given by
\begin{equation} \label{conservationlaw}
\partial_t y(t,x) + A  \partial_x  y(t,x) = 0 
\end{equation}
with time and space variables $(t,x)\in [0,\infty)\times [0,1]$, and $y(t,x) \in \R^2$. Let $A\in \mathbb{R}^{2\times 2}$ be a diagonalisable matrix with real eigenvalues,
\begin{equation}\label{distinctEW}
{a}^- < 0 < {a}^+,
\end{equation}
 separated by zero. The steady state,  $\bar{y}(x)$, of the system is zero.  For such a $2\times 2$ strictly hyperbolic system there exists a diagonal matrix 
$
\bar{\Lambda} := \text{diag}\big\{{a}^+,{a}^-\big\}
$, 
whose entries are eigenvalues of $A$ and a matrix $E$ whose rows are left eigenvectors of the matrix $A$.

We further introduce the Riemann invariants (see Section 7.3 in \cite{C.M.Dafermos:2010aa})   
\begin{equation}\label{RiemannCoordinates} 
U(t,x):= \begin{pmatrix} U^+(t,x) \\ U^-(t,x) \end{pmatrix} := Ey(t,x) \in \R^2.
\end{equation}
Any smooth solutions to equation \eqref{conservationlaw} also solve the system 
\begin{equation}\label{linSystem}
\partial_t U(t,x) + \bar{\Lambda}  \partial_x  U(t,x) = 0,\; x \in [0,1], t \geq 0.
\end{equation}
We prescribe initial conditions $U_0 \in C^2(0,1)$ 
\begin{equation} \label{ic} U(0,x) = U_0(x), \;x \in [0,1], \end{equation} 
and linear feedback boundary conditions given by 
\begin{equation}\label{bc}
\begin{aligned}
\begin{pmatrix}
	U^+(t,0) \\ U^-(t,1)
\end{pmatrix}
&=K \;
\begin{pmatrix}
	U^+(t,1) \\ U^-(t,0)
\end{pmatrix}.
\end{aligned}
\end{equation}
where a matrix $K \in \mathbb{R}^{2\times 2}$.
Due to assumption~\eqref{distinctEW} characteristics are time-independent and can not intersect. Therefore, a solution $U(t,x)$ to equations \eqref{linSystem}, \eqref{ic} and \eqref{bc} exists for all $t\in \mathbb{R}^+_0$ as discussed in \cite{C.M.Dafermos:2010aa}.   We will show that $U(t,x)$ is stable with respect to the $L^2$-norm in the sense of Definition \ref{DefL2Stability} below, which is taken from \cite[Sec.~3]{Coron_2009}, where the precise definition of an $L^2$-solution is also found.

\begin{defn}\label{DefL2Stability}
The system \eqref{linSystem}, \eqref{ic} and \eqref{bc}  is exponentially stable with respect to  the $L^2$-norm if there exist constants $C_1>0$ and $C_2>0$ such that for every $U_0(\cdot) \in L^2\big( [0,1];\mathbb{R}^2 \big)$ the $L^2$-solution of equations \eqref{linSystem}, \eqref{ic} and \eqref{bc} satisfies
\begin{equation}\label{decay}
\big\lVert U(t,\cdot) \big\rVert_{L^2( [0,1];\mathbb{R}^2 )}
 \leq C_2 e^{-C_1 t}\big\lVert U(0,\cdot) \big\rVert_{L^2( [0,1];\mathbb{R}^2 )}, \quad
\forall \ t \in \mathbb{R}^+_0.
\end{equation}
\end{defn}
Several conditions on the matrix $K$ have been discussed such that  exponential stability in the sense
of Definition \ref{DefL2Stability} holds, see e.g. \cite{J.-M.Coron:2015aa} and references therein. 

We are interested in decay rates $C_1$ of numerical approximations to $U(t,x)$ on a spatial and temporal computational grid. Using a spatial discretisation width $\Delta x$, the space interval $[0,1]$ is divided into $J$ cells such that $\Delta x J = 1$ with cell centres $x_j := \big(j-\frac{1}{2}\big) \Delta x$  for $j=1,\ldots,J$ and outside the computational domain ghost-cells with centres $x_0$ and $x_{J+1}$ are added. The discrete time is denoted by $t^n:= n \Delta t$ for $n = 0,1 \ldots$, where $\Delta t>0$ such that the CFL-condition
\begin{equation}\label{CFL}
\textup{CFL} := \max \{ {a}^+, |{a}^-| \} 
\frac{\Delta t}{\Delta x} \leq 1,
\end{equation}
holds. A numerical finite--volume scheme  approximates the cell averages  of the solution $U(t,\cdot)$ at $t^n$  by
\begin{equation*}
U_j^n:= 
\begin{pmatrix} U_j^{n,+} \\ U_j^{n,-} \end{pmatrix} 
\approx \frac{1}{\Delta x}
\int\limits_{x_{j-1/2}}^{x_{j+1/2}} 
\begin{pmatrix} U^+(t^n,x) \\ U^-(t^n,x) \end{pmatrix} 
d x \in \mathbb{R}^2.
\end{equation*}
For $\dx$ sufficiently small,  a piece-wise constant reconstruction of $\mathbf{U}$ 
\begin{equation}
\mathbf{U}(t,x) = \sum\limits_{\substack{n = 0,1, \ldots \\ j=1, \ldots, J}} \chi_{ [t^n,t^{n+1}]\times [x_{j-1/2}, x_{j+1/2}]} U_j^n,
\end{equation} 
is expected to also satisfy the exponential stability criterion in equation \eqref{decay}. 

Given a numerical scheme, we are therefore interested in the question of how the numerical discretisation error deteriorates the theoretical expected rate $C_1$ in equation \eqref{decay}. As discussed in the introduction, the decay rates observed for $\mathbf{U}$ computed by numerical schemes are higher compared with the theoretical prediction. It has been conjectured \cite{M.K.Banda:2013aa} that the additional decay is due to artificial viscosity present in numerical schemes. In order to understand the role of possible viscosity, we present the following formal computation before turning to an analysis of the finite volume scheme for the computation of $\mathbf{U}.$ 

\begin{lem} \label{lem:contsStab}
For $\epsilon>0$,  given constant matrices $\bar{\Lambda}$ and $K$ and smooth initial data $V_0(\cdot)$, assume a $C^2( [0,\infty) \times [0,1] )^2$ solution $V(t,x)=(V^+,V^-)(t,x)$ of the viscous hyperbolic system 
\begin{align} \label{diffcont} 
&\partial_t V(t,x) +  \bar{\Lambda}  \partial_x  V(t,x)  =  \epsilon \partial_{xx} V(t,x), \\
&V(0,x) = V_0(x),\label{initialcont} \\
&\begin{pmatrix} V^+(t,0) \\ V^-(t,1) \end{pmatrix}
=K \; \begin{pmatrix}
	V^+(t,1) \\ V^-(t,0)
\end{pmatrix}, \;  \partial_x \begin{pmatrix}
	V^+(t,0) \\ V^-(t,1)
\end{pmatrix}
=K \;\partial_x 
\begin{pmatrix}
	V^+(t,1) \\ V^-(t,0)
\end{pmatrix},\label{boundarycont}
\end{align}
exists. If $K$ satisfies the following conditions 
\begin{align}
K^\top \begin{pmatrix} a^+ - \epsilon \mu & 0 \\ 0 & (|a^-| - \epsilon \mu) e^{\mu}  \end{pmatrix} K  =&\; \begin{pmatrix} (a^+ - \epsilon \mu)e^{-\mu} & 0 \\ 0 & |a^-| - \epsilon \mu \end{pmatrix},\label{eq:ThmCond01}\\
K^\top \begin{pmatrix} 1 & 0 \\ 0 & -e^{\mu}  \end{pmatrix} K =&\; \begin{pmatrix} e^{-\mu} & 0 \\ 0 & -1 \end{pmatrix},\label{eq:ThmCond02}
\end{align}
then $V$ is exponentially stable in the sense of Definition \ref{DefL2Stability}.
% with rate
%\begin{equation} 
%C_1 = \alpha \mu  - \epsilon \mu^2,
%\end{equation}
%where $ \alpha := \min \{ a^+,\; |a^-|\} $. 
\end{lem}

\begin{proof}
For a fixed positive value $\mu>0$, we consider the non--negative candidate Lyapunov function $\mathcal{L}$ defined
on $[0, +\infty)$ for a $C^2-$solution of equation \eqref{diffcont} 
\begin{equation}\label{eq:LyapFunSystem}
\mathcal{L}(t) = \int_0^1 V^\top(t,x) P(\mu x) V(t,x) \mathrm{d}x, 
\end{equation}
where $ P(\mu x) := \text{diag}\{ e^{-\mu x},\;e^{\mu x}\} $.  

Since $\mathcal{L}$ is differentiable, we use integration by parts to analyse the time derivative of the Lyapunov function 
\begin{align*} 
\frac{d}{dt} \mathcal{L}(t) =& \int_0^1 2V^\top(t,x) P(\mu x) \partial_t V(t,x) \mathrm{d}x,\\ 
=& - 2\int_0^1 V^\top(t,x) P(\mu x)\bar{\Lambda} \partial_xV(t,x) \mathrm{d}x + 2\epsilon \int_0^1 V^\top(t,x) P(\mu x) \partial_{xx} V(t,x) \mathrm{d}x, \\
=& - \mu \int_0^1 V^\top P(\mu x)|\bar{\Lambda}| V \mathrm{d}x - 2\epsilon \int_0^1 \partial_xV^\top P(\mu x) \partial_xV \mathrm{d}x + \mu^2 \epsilon \int_0^1 V^\top P(\mu x) V \mathrm{d}x\\
 &- \left[V^\top(t,x) \left(P(\mu x)\bar{\Lambda} + \epsilon \dfrac{d}{\mathrm{d}x}P(\mu x) \right) V(t,x) \right]\bigg|_{x=0}^{x=1} + 2\epsilon \left[V^\top P(\mu x) \partial_xV \right]\Big|_{x=0}^{x=1}.
\end{align*}

By using the boundary conditions \eqref{boundarycont}, we have 
\begin{align*} 
\frac{d}{dt} \mathcal{L}(t) =&  - \mu \int_0^1 V^\top P(\mu x)|\bar{\Lambda}| V \mathrm{d}x - 2\epsilon \int_0^1 \partial_xV^\top P(\mu x) \partial_xV \mathrm{d}x + \mu^2 \epsilon \int_0^1 V^\top P(\mu x) V \mathrm{d}x\\
& - \begin{pmatrix} V^+(t,1) \\ V^-(t,0) \end{pmatrix}^\top \left( \begin{pmatrix} (a^+ - \epsilon \mu)e^{-\mu} & 0 \\ 0 & |a^-| - \epsilon \mu \end{pmatrix} - K^\top \begin{pmatrix} a^+ - \epsilon \mu & 0 \\ 0 & (|a^-| - \epsilon \mu) e^{\mu}  \end{pmatrix} K \right) \begin{pmatrix} V^+(t,1) \\ V^-(t,0) \end{pmatrix}\\
& + 2\epsilon\begin{pmatrix} V^+(t,1) \\ V^-(t,0) \end{pmatrix}^\top \left( \begin{pmatrix} e^{-\mu} & 0 \\ 0 & -1 \end{pmatrix} - K^\top \begin{pmatrix} 1 & 0 \\ 0 & -e^{\mu}  \end{pmatrix} K \right) \partial_x\begin{pmatrix} V^+(t,1) \\ V^-(t,0) \end{pmatrix},\\  
\end{align*}
 
Let $ \alpha:= \min\{ a^+,\; |a^-|\} $. Then, by \eqref{eq:ThmCond01}, \eqref{eq:ThmCond02}, the time derivative of the Lyapunov function is estimated by 
\begin{align*} 
\frac{d}{dt} \mathcal{L}(t) \leq & - \alpha\mu \int_0^1 V^\top P(\mu x) V \mathrm{d}x  +  \mu^2\epsilon \int_0^1 V^\top P(\mu x) V \mathrm{d}x = -\eta_T\mathcal{L}(t),
\end{align*}
where \begin{equation}\label{eqn:eta_T} \eta_T:=   \alpha\mu - \epsilon\mu^2. \end{equation} Hence. 
\begin{equation*}
\mathcal{L}(t) \leq e^{-\eta_T t}\mathcal{L}(0),\; t \in [0, \infty). 
\end{equation*}
We also note that $\mathcal{L}$ can be bound in the $L^2$-norm since $P(\mu x)$ is positive definite, it can be proved that $\lambda_{\min}(P(\mu x)) \|U\|_{L^2}^2 \le \mathcal{L}(t) \le \lambda_{\max}(P(\mu x))\|U\|_{L^2}^2$, where $\lambda_{\min}(\cdot) $ and $\lambda_{\max}(\cdot)$ denote the smallest and largest eigenvalues, respectively. See \cite{Bastin:2016aa, Weldegiyorgis_2019} for details.
\end{proof}

\begin{rem}
	In the case $\epsilon=0$, we recover known results on the stability. More general conditions \eqref{eq:ThmCond01}, \eqref{eq:ThmCond02} on $K$ in the case $\epsilon=0$ exist and we refer to  \cite{J.-M.Coron:2008aa} for more details. Later, the value of $\epsilon$ is related to artificially introduced numerical diffusion due to a spatial discretisation \cite{LeVeque:2002aa}. 
\end{rem}

In the next section, a presentation of the analysis of the discretised system will be presented. The section presents the main results of this paper.
%--------------------------------------------------------------------------------------------
%--------------------------------------------------------------------------------------------

\section{Main results}\label{sec2}

We consider the linear hyperbolic system of conservation laws \eqref{linSystem} and use a first order explicit Upwind scheme to approximate the derivatives. Thus, we obtain the following approximation of the system \eqref{linSystem} at point $ x_{j} $ and time $ t^n $
\begin{equation}\label{eq:disclinsystem}
\frac{1}{\Delta t} \begin{pmatrix} U_{j}^{n+1,+} - U_{j}^{n,+}  \\ U_{j}^{n+1,-} - U_{j}^{n,-} \end{pmatrix} + \frac{1}{\Delta x} \bar{\Lambda}\begin{pmatrix} U_{j}^{n,+} - U_{j-1}^{n,+}  \\ U_{j+1}^{n,-} - U_{j}^{n,-} \end{pmatrix} = 0,\quad j = 1,\ldots,J,\; n = 0,1,\ldots. 
\end{equation}

By using Taylor expansion on the approximation and taking further derivatives of the Taylor expansion with respect to time and space, we obtain the final form of the modified equation as (see Section 1 in \cite{Carpentier_1997} for details of the derivation): 
\begin{equation}\label{eq:ModifiedlinSystem}
\partial_t U(t,x) + \bar{\Lambda}  \partial_x  U(t,x) = \frac{1}{2}{\Delta x}|\bar{\Lambda}|\left( I_{2} - \frac{\Delta t}{\Delta x}|\bar{\Lambda}| \right) \partial_{xx}  U(t,x) + \mathcal{O}\left({\Delta t}, {\Delta x}  \right)^2,\; x \in [0,1], t \geq 0.
\end{equation}

\begin{rem}\label{rem:DiffusionConstant}
The diffusion coefficients $\epsilon^+ := \displaystyle{\frac{1}{2} a^+\dx\Big(1 - a^+\frac{\dt}{\dx}\Big)}$ and $\epsilon^- := \displaystyle{\frac{1}{2} |a^-|\dx\Big(1 - |a^-|\frac{\dt}{\dx}\Big)}$ vanish for $a^+\dt = \dx$ and $|a^-|\dt = \dx$, respectively. In such a case, the advection equation is solved exactly by the Upwind method. The diffusion coefficient is positive for $0 < |a^\pm|\dt/\dx < 1$, as a result the Upwind method is also stable.
\end{rem}

On the discretised space and time intervals, the discretisation of the modified equations \eqref{eq:ModifiedlinSystem} by using the Upwind scheme for first-order spatial derivatives, the Euler scheme for first-order time derivatives and second-order finite difference scheme for second-order spatial derivatives, we obtain 
\begin{subequations}\label{eq:discModified}
\begin{align}
\begin{pmatrix} U_{j}^{n+1,+} \\ U_{j}^{n+1,-} \end{pmatrix}=\; \begin{pmatrix} U_{j}^{n,+}  \\ U_{j}^{n,-} \end{pmatrix} - \frac{\Delta t}{\Delta x}\bar{\Lambda} \begin{pmatrix} U_{j}^{n,+} - U_{j-1}^{n,+}  \\ U_{j+1}^{n,-} - U_{j}^{n,-} \end{pmatrix} + \frac{\Delta t}{{\Delta x}^2}\begin{pmatrix} \epsilon^+\left(U_{j+1}^{n,+} - 2U_{j}^{n,+} + U_{j-1}^{n,+}\right)  \\ \epsilon^-\left(U_{j+1}^{n,-} -2U_{j}^{n,-} + U_{j-1}^{n,-}\right) \end{pmatrix},&\label{eq:discModifiedlinsystem}\\
 j = 1, \ldots, J,\; n = 0,1 \ldots,& \nonumber\\
U_{j}^{0} =\; \begin{pmatrix} U_{j}^{0,+} \\ U_{j}^{0,-} \end{pmatrix},\qquad \qquad\qquad\qquad\qquad\qquad\qquad\qquad\qquad\qquad j = 1, \ldots, J,&\label{eq:discModifiedlinsystemIC}\\
\begin{pmatrix} U_{0}^{n,+} \\ U_{J+1}^{n,-} \end{pmatrix}=\; K\begin{pmatrix} U_{J}^{n,+}  \\ U_{1}^{n,-} \end{pmatrix},\;  \begin{pmatrix} U_{1}^{n,+} - U_{0}^{n,+} \\ U_{J+1}^{n,-} - U_{J}^{n,-} \end{pmatrix}= K\begin{pmatrix} U_{J+1}^{n,+} - U_{J}^{n,+}  \\ U_{1}^{n,-} - U_{0}^{n,-} \end{pmatrix}, \quad n = 0,1, \ldots.&\label{eq:discModifiedlinsystemBCs}   
\end{align}
\end{subequations}
Below we present a definition of the discrete Lyapunov function as well as discrete exponential stability which will be followed by the main results of this section:
\begin{defn}[Discrete Lyapunov function]
For $ \mu > 0 $, we define a discrete Lyapunov function by
\begin{equation}\label{eq:dLyafun-system}
\mathcal{L}^{n} = {\Delta x} \sum_{j=1}^{J} U_{j}^{n \top} P_j(\mu) U_{j}^{n}, \quad n = 0,1, \ldots,
\end{equation}
where $ P_j(\mu):= \text{diag} \{e^{-\mu x_j}, \; e^{\mu x_j}\} $. 
\end{defn}

\begin{defn}[Discrete Exponential Stability]\label{Def:DiscExp}
The discretised system \eqref{eq:discModifiedlinsystem}, initial condition \eqref{eq:discModifiedlinsystemIC} and boundary conditions \eqref{eq:discModifiedlinsystemBCs} is exponentially stable with respect to the $ L^2- $norm if there exist $C_1 > 0$ and $C_2 > 0$ such that for every $U_j^0 $, $j = 1, \ldots , J$, the discrete $ L^2- $solution satisfies 
\begin{equation}\label{eq:DiscreteExpStability}
\|U_j^n \|_{L_{\Delta x}^2} \leq C_2e^{-C_1 t^n} \|U_j^0 \|_{L_{\Delta x}^2}, \quad n = 0,1, \ldots,
\end{equation}
where 
\begin{equation*}
\|U_j^n \|_{L_{\Delta x}^2}:=\sqrt{{\Delta x} \sum_{j=1}^{J} \left(|U_j^{n,+}|^2 + |U_j^{n,-}|^2\right) },\quad  \|U_j^0 \|_{L_{\Delta x}^2}:=\sqrt{{\Delta x} \sum_{j=1}^{J} \left(|U_j^{0,+}|^2 + |U_j^{0,-}|^2\right) }. 
\end{equation*}
\end{defn}

\begin{thm}[Stability]\label{NumericalStabThem}
Assume that the CFL condition \eqref{CFL} holds. Let $ \epsilon = \max\{\epsilon^+,\; \epsilon^- \}$. For $ \mu > 0 $ and matrix $ K $ satisfying the conditions \eqref{eq:NumThm01} - \eqref{eq:NumThm03}, if 
\begin{equation}\label{eq:NumThm01}
K^\top \left(\begin{smallmatrix}\left(a^+ + \frac{\epsilon^+}{{\Delta x}}\left(e^{-\mu {\Delta x}} -1\right)e^{\mu {\Delta x}} \right)e^{-\mu x_{1}} & 0 \\ 0 &  \left(|a^-| + \frac{\epsilon^-}{{\Delta x}}\left(e^{-\mu {\Delta x}} -1\right)e^{\mu {\Delta x}} \right)e^{\mu x_{J}} \end{smallmatrix} \right)K =  \left(\begin{smallmatrix}\left(a^+ + \frac{\epsilon^+}{{\Delta x}}\left(e^{-\mu {\Delta x}} -1\right)e^{\mu {\Delta x}} \right)e^{-\mu x_{J+1}} & 0 \\ 0 &  \left(|a^-| + \frac{\epsilon^-}{{\Delta x}}\left(e^{-\mu {\Delta x}} -1\right)e^{\mu {\Delta x}} \right)e^{\mu x_{0}} \end{smallmatrix} \right),
\end{equation} 
\begin{equation}\label{eq:NumThm02}
K^\top \left(\begin{smallmatrix} \frac{\epsilon^+}{{\Delta x}}e^{-\mu x_{0}} & 0 \\ 0 &  -\frac{\epsilon^-}{{\Delta x}}e^{\mu x_{J+1}} \end{smallmatrix} \right) K = \left(\begin{smallmatrix} \frac{\epsilon^+}{{\Delta x}}e^{-\mu x_{J}} & 0 \\ 0 &  -\frac{\epsilon^-}{{\Delta x}}e^{\mu x_{1}} \end{smallmatrix} \right),  
\end{equation}
and
\begin{equation}\label{eq:NumThm03}
K^\top \left(\begin{smallmatrix} \frac{\epsilon^+}{{\Delta x}}\left(a^+ + 2\frac{\epsilon^+}{\Delta x} \right) e^{-\mu x_{0}} & 0 \\ 0 &  -\frac{\epsilon^-}{{\Delta x}}\left(|a^-| + 2\frac{\epsilon^-}{\Delta x} \right)e^{\mu x_{J+1}} \end{smallmatrix} \right)K  = \left(\begin{smallmatrix} \frac{\epsilon^+}{{\Delta x}}\left(a^+ + 2\frac{\epsilon^+}{\Delta x} \right) e^{-\mu x_{J}} & 0 \\ 0 & - \frac{\epsilon^-}{{\Delta x}}\left(|a^-| + 2\frac{\epsilon^-}{\Delta x} \right)e^{\mu x_{1}} \end{smallmatrix} \right),  
\end{equation}
then the numerical solution $ U_{j}^{n} $ of the system \eqref{eq:discModified} is exponentially stable in the sense of Definition \ref{Def:DiscExp}.% with decay rate
%\begin{equation*}
%C_1 = \alpha \mu e^{-\mu {\Delta x}} - \epsilon\mu^2 > 0.
%\end{equation*}
\end{thm}
\begin{proof}
For a fixed $ \mu > 0 $, the discrete Lyapunov function defined by \eqref{eq:dLyafun-system} is estimated by 
\begin{equation}\label{eq:DiscNormEquivalence}
e^{-\mu}\|U_j^n \|_{L_{\Delta x}^2}^2 \leq \mathcal{L}^n \leq e^{\mu}\|U_j^n \|_{L_{\Delta x}^2}^2, \quad n = 0,1, \ldots. 
\end{equation}	
From using the discrete candidate Lyapunov function defined by \eqref{eq:dLyafun-system}, the time derivative of the candidate Lyapunov function \eqref{eq:LyapFunSystem} is approximated by  
\begin{equation}\label{eq:AppxDtLyafun01}
\begin{split}
\frac{\mathcal{L}^{n+1} - \mathcal{L}^{n}}{\Delta t} =&\; \frac{\Delta x}{\Delta t} \sum_{j=1}^{J} \left(U_{j}^{n+1 \top} P_j(\mu) U_{j}^{n+1} - U_{j}^{n \top} P_j(\mu) U_{j}^{n}\right), \quad n = 0,1 \ldots,\\ 
=&\; \underbrace{\frac{\Delta x}{\Delta t} \sum_{j=1}^{J} \left(\left( U_{j}^{n+1,+} \right)^2 - \left( U_{j}^{n,+} \right)^2\right)e^{-\mu x_j}}_{D_{j}^{n,+}} + \underbrace{\frac{\Delta x}{\Delta t} \sum_{j=1}^{J}\left(\left( U_{j}^{n+1,-} \right)^2 - \left( U_{j}^{n,-} \right)^2\right)e^{\mu x_j}}_{D_{j}^{n,-}}.
\end{split}\end{equation}

By the discrete system \eqref{eq:discModifiedlinsystem}, we have 
\begin{equation}\label{eq:AppxDtLyafun002}
\begin{split}
D_{j}^{n,+} =&\; \frac{\Delta x}{\Delta t} \sum_{j=1}^{J}\left[U_{j}^{n,+}  - \left(\frac{\Delta t}{\Delta x}a^+ + \frac{\Delta t}{{\Delta x}^2}\epsilon^+ \right)\left(U_{j}^{n,+} - U_{j-1}^{n,+}\right) + \left(\frac{\Delta t}{{\Delta x}^2}\epsilon^+ \right)\left(U_{j+1}^{n,+} - U_{j}^{n,+}\right)\right]^2e^{-\mu x_j},\\
&\; - \frac{\Delta x}{\Delta t} \sum_{j=1}^{J} \left(U_{j}^{n,+}\right)^2e^{-\mu x_j},\\
=&\; -2\left(\frac{\Delta t}{\Delta x}a^+ + \frac{\Delta t}{{\Delta x}^2}\epsilon^+ \right)\frac{\Delta x}{\Delta t} \sum_{j=1}^{J}\left(U_{j}^{n,+}\right)\left(U_{j}^{n,+} - U_{j-1}^{n,+}\right)e^{-\mu x_j}\\
&\; + \left(\frac{\Delta t}{\Delta x}a^+ + \frac{\Delta t}{{\Delta x}^2}\epsilon^+ \right)^2\frac{\Delta x}{\Delta t} \sum_{j=1}^{J}\left(U_{j}^{n,+} - U_{j-1}^{n,+}\right)^2e^{-\mu x_j}\\
&\; + \left(\frac{\Delta t}{{\Delta x}^2}\epsilon^+ \right)^2\frac{\Delta x}{\Delta t} \sum_{j=1}^{J}\left(U_{j+1}^{n,+} - U_{j}^{n,+}\right)^2e^{-\mu x_j} \\
&\; + 2\left(\frac{\Delta t}{{\Delta x}^2}\epsilon^+ \right)\frac{\Delta x}{\Delta t} \sum_{j=1}^{J}\left(U_{j}^{n,+}\right)\left(U_{j+1}^{n,+} - U_{j}^{n,+}\right)e^{-\mu x_j}\\
&\;-2\left(\frac{\Delta t}{{\Delta x}^2}\epsilon^+ \right)\left(\frac{\Delta t}{\Delta x}a^+ + \frac{\Delta t}{{\Delta x}^2}\epsilon^+ \right)\frac{\Delta x}{\Delta t} \sum_{j=1}^{J}\left(U_{j}^{n,+} - U_{j-1}^{n,+}\right)\left(U_{j+1}^{n,+} - U_{j}^{n,+}\right)e^{-\mu x_j},
\end{split}
\end{equation} 
and 
\begin{equation}\label{eq:AppxDtLyafun003}
\begin{split}
D_{j}^{n,-} =&\; \frac{\Delta x}{\Delta t} \sum_{j=1}^{J}\left[U_{j}^{n,-}  + \left(\frac{\Delta t}{\Delta x}|a^-| + \frac{\Delta t}{{\Delta x}^2}\epsilon^- \right)\left(U_{j+1}^{n,-} - U_{j}^{n,-}\right) - \left(\frac{\Delta t}{{\Delta x}^2}\epsilon^- \right)\left(U_{j}^{n,-} - U_{j-1}^{n,-}\right)\right]^2 e^{\mu x_j}\\
&\; -\frac{\Delta x}{\Delta t} \sum_{j=1}^{J}\left(U_{j}^{n,-}\right)^2e^{\mu x_j},\\
=&\;+2\left(\frac{\Delta t}{\Delta x}|a^-| + \frac{\Delta t}{{\Delta x}^2}\epsilon^- \right)\frac{\Delta x}{\Delta t} \sum_{j=1}^{J}\left(U_{j}^{n,-}\right)\left(U_{j+1}^{n,-} - U_{j}^{n,-}\right)e^{\mu x_j}\\
&\; + \left(\frac{\Delta t}{\Delta x}|a^-| + \frac{\Delta t}{{\Delta x}^2}\epsilon^- \right)^2\frac{\Delta x}{\Delta t} \sum_{j=1}^{J}\left(U_{j+1}^{n,-} - U_{j}^{n,-}\right)^2e^{\mu x_j}\\
&\; + \left(\frac{\Delta t}{{\Delta x}^2}\epsilon^- \right)^2\frac{\Delta x}{\Delta t} \sum_{j=1}^{J}\left(U_{j}^{n,-} - U_{j-1}^{n,-}\right)^2 e^{\mu x_j}\\
&\; - 2\left(\frac{\Delta t}{{\Delta x}^2}\epsilon^- \right)\frac{\Delta x}{\Delta t} \sum_{j=1}^{J}\left(U_{j}^{n,-}\right)\left(U_{j}^{n,-} - U_{j-1}^{n,-}\right)e^{\mu x_j}\\
&\;-2\left(\frac{\Delta t}{{\Delta x}^2}\epsilon^- \right)\left(\frac{\Delta t}{\Delta x}|a^-| + \frac{\Delta t}{{\Delta x}^2}\epsilon^- \right)\frac{\Delta x}{\Delta t} \sum_{j=1}^{J}\left(U_{j+1}^{n,-} - U_{j}^{n,-}\right)\left(U_{j}^{n,-} - U_{j-1}^{n,-}\right)e^{\mu x_j}.
\end{split}
\end{equation}

By using Young's inequality (i.e. $ \pm 2ab \leq a^2 + b^2\; a,b \in \mathbb{R}$), $ x_{j+1} = x_{j} + {\Delta x},\; j = 0,\ldots,J $, and the CFL condition \eqref{CFL} in \eqref{eq:AppxDtLyafun002},\eqref{eq:AppxDtLyafun003}, we obtain:
\begin{equation}\label{eq:AppxDtLyafun004}
\begin{split}
D_{j}^{n,+} \leq&\; -2\left(\frac{\Delta t}{\Delta x}a^+ + \frac{\Delta t}{{\Delta x}^2}\epsilon^+ \right)\frac{\Delta x}{\Delta t} \sum_{j=1}^{J}\left(U_{j}^{n,+}\right)\left(U_{j}^{n,+} - U_{j-1}^{n,+}\right)e^{-\mu x_j}\\
&\;+\left(\frac{\Delta t}{\Delta x}a^+ + \frac{\Delta t}{{\Delta x}^2}\epsilon^+ \right)\left(\frac{\Delta t}{\Delta x}a^+ + 2\frac{\Delta t}{{\Delta x}^2}\epsilon^+ \right)\frac{\Delta x}{\Delta t} \sum_{j=1}^{J}\left(U_{j}^{n,+} - U_{j-1}^{n,+}\right)^2e^{-\mu x_j}\\
&\;+2e^{\mu {\Delta x}}\left(\frac{\Delta t}{{\Delta x}^2}\epsilon^+ \right)\frac{\Delta x}{\Delta t} \sum_{j=1}^{J}\left(U_{j}^{n,+}\right)\left(U_{j}^{n,+} - U_{j-1}^{n,+}\right)e^{-\mu x_j}\\
&\;-e^{\mu {\Delta x}}\left(\frac{\Delta t}{{\Delta x}^2}\epsilon^+ \right)\left(2 - \frac{\Delta t}{\Delta x}a^+ - 2\frac{\Delta t}{{\Delta x}^2}\epsilon^+ \right)\frac{\Delta x}{\Delta t} \sum_{j=1}^{J}\left(U_{j}^{n,+} - U_{j-1}^{n,+}\right)^2e^{-\mu x_j}\\
&\; - 2\left(\frac{\Delta t}{{\Delta x}^2}\epsilon^+ \right)\frac{\Delta x}{\Delta t} \left[\left(U_{0}^{n,+}\right)\left(U_{1}^{n,+} - U_{0}^{n,+}\right)e^{-\mu x_0} - \left(U_{J}^{n,+}\right)\left(U_{J+1}^{n,+} - U_{J}^{n,+}\right)e^{-\mu x_J}\right]\\
&\;- \left(\frac{\Delta t}{{\Delta x}^2}\epsilon^+ \right)\left(\frac{\Delta t}{\Delta x}a^+ + 2\frac{\Delta t}{{\Delta x}^2}\epsilon^+ \right)\frac{\Delta x}{\Delta t} \left[\left(U_{1}^{n,+} - U_{0}^{n,+}\right)^2e^{-\mu x_0} - \left(U_{J+1}^{n,+} - U_{J}^{n,+}\right)^2e^{-\mu x_J}\right],
\end{split}
\end{equation} 
and 
\begin{equation}\label{eq:AppxDtLyafun005}
\begin{split}
D_{j}^{n,-} \leq&\; +2\left(\frac{\Delta t}{\Delta x}|a^-| + \frac{\Delta t}{{\Delta x}^2}\epsilon^- \right)\frac{\Delta x}{\Delta t} \sum_{j=1}^{J}\left(U_{j}^{n,-}\right)\left(U_{j+1}^{n,-} - U_{j}^{n,-}\right)e^{\mu x_j}\\
&\;+\left(\frac{\Delta t}{\Delta x}|a^-| + \frac{\Delta t}{{\Delta x}^2}\epsilon^- \right)\left(\frac{\Delta t}{\Delta x}|a^-| + 2\frac{\Delta t}{{\Delta x}^2}\epsilon^- \right)\frac{\Delta x}{\Delta t} \sum_{j=1}^{J}\left(U_{j+1}^{n,-} - U_{j}^{n,-}\right)^2e^{\mu x_j}\\
&\; - 2e^{\mu {\Delta x}}\left(\frac{\Delta t}{{\Delta x}^2}\epsilon^- \right)\frac{\Delta x}{\Delta t} \sum_{j=1}^{J}\left(U_{j}^{n,-}\right)\left(U_{j+1}^{n,-} - U_{j}^{n,-}\right)e^{\mu x_j}\\
&\;-e^{\mu {\Delta x}}\left(\frac{\Delta t}{{\Delta x}^2}\epsilon^- \right)\left(2 - \frac{\Delta t}{\Delta x}|a^-| - 2\frac{\Delta t}{{\Delta x}^2}\epsilon^- \right)\frac{\Delta x}{\Delta t} \sum_{j=1}^{J}\left(U_{j+1}^{n,-} - U_{j}^{n,-}\right)^2e^{\mu x_j}\\
&\; + 2\left(\frac{\Delta t}{{\Delta x}^2}\epsilon^- \right)\frac{\Delta x}{\Delta t} \left[\left(U_{J+1}^{n,-}\right)\left(U_{J+1}^{n,-} - U_{J}^{n,-}\right)e^{\mu x_{J+1}} - \left(U_{1}^{n,-}\right)\left(U_{1}^{n,-} - U_{0}^{n,-}\right)e^{\mu x_1} \right]\\
&\;-\left(\frac{\Delta t}{{\Delta x}^2}\epsilon^- \right)\left(\frac{\Delta t}{\Delta x}|a^-| + 2\frac{\Delta t}{{\Delta x}^2}\epsilon^- \right)\frac{\Delta x}{\Delta t} \left[\left(U_{J+1}^{n,-} - U_{J}^{n,-}\right)^2e^{\mu x_{J+1}} - \left(U_{1}^{n,-} - U_{0}^{n,-}\right)^2e^{\mu x_1}\right].
\end{split}
\end{equation}

By using $ 2a(a -b) =  a^2 - b^2 + (a - b)^2,\; a,b \in \mathbb{R}$ repeatedly, and  $ x_{j+1} = x_{j} + {\Delta x},\; j = 0,\ldots,J $ in \eqref{eq:AppxDtLyafun004}, \eqref{eq:AppxDtLyafun005}, we have 
\begin{equation}\label{eq:AppxDtLyafun006}
\begin{split}
D_{j}^{n,+} \leq&\; -\left[\left(\frac{\Delta t}{\Delta x}a^+ +2\frac{\Delta t}{{\Delta x}^2}\epsilon^+ \right) - \left(\frac{\Delta t}{\Delta x}a^+ +\frac{\Delta t}{{\Delta x}^2}\epsilon^+ \right)e^{-\mu {\Delta x}} - \left(\frac{\Delta t}{{\Delta x}^2}\epsilon^+ \right)e^{\mu {\Delta x}}\right]\frac{\Delta x}{\Delta t} \sum_{j=1}^{J}\left(U_{j}^{n,+}\right)e^{-\mu x_j}\\
&\;-\left(\frac{\Delta t}{\Delta x}a^+ + \frac{\Delta t}{{\Delta x}^2}\epsilon^+ \right)\left(1 - \frac{\Delta t}{\Delta x}a^+ - 2\frac{\Delta t}{{\Delta x}^2}\epsilon^+ \right)\frac{\Delta x}{\Delta t} \sum_{j=1}^{J}\left(U_{j}^{n,+} - U_{j-1}^{n,+}\right)^2e^{-\mu x_j}\\
&\;-e^{\mu {\Delta x}}\left(\frac{\Delta t}{{\Delta x}^2}\epsilon^+ \right)\left(1 - \frac{\Delta t}{\Delta x}a^+ - 2\frac{\Delta t}{{\Delta x}^2}\epsilon^+ \right)\frac{\Delta x}{\Delta t} \sum_{j=1}^{J}\left(U_{j}^{n,+} - U_{j-1}^{n,+}\right)^2e^{-\mu x_j}\\
&\; + \left(\frac{\Delta t}{\Delta x}a^+e^{-\mu {\Delta x}} + \left(e^{-\mu {\Delta x}} - 1\right)\left(\frac{\Delta t}{{\Delta x}^2}\epsilon^+ \right)\right)\frac{\Delta x}{\Delta t} \left[\left(U_{0}^{n,+}\right)^2e^{-\mu x_0} - \left(U_{J}^{n,+}\right)^2e^{-\mu x_J}\right]\\
&\; + 2\left(\frac{\Delta t}{{\Delta x}^2}\epsilon^+ \right)\frac{\Delta x}{\Delta t} \left[\left(U_{0}^{n,+}\right)\left(U_{1}^{n,+} - U_{0}^{n,+}\right)e^{-\mu x_0} - \left(U_{J}^{n,+}\right)\left(U_{J+1}^{n,+} - U_{J}^{n,+}\right)e^{-\mu x_J}\right]\\
&\;- \left(\frac{\Delta t}{{\Delta x}^2}\epsilon^+ \right)\left(\frac{\Delta t}{\Delta x}a^+ + 2\frac{\Delta t}{{\Delta x}^2}\epsilon^+ \right)\frac{\Delta x}{\Delta t} \left[\left(U_{1}^{n,+} - U_{0}^{n,+}\right)^2e^{-\mu x_0} - \left(U_{J+1}^{n,+} - U_{J}^{n,+}\right)^2e^{-\mu x_J}\right],
\end{split}
\end{equation} 
and 
\begin{equation}\label{eq:AppxDtLyafun007}
\begin{split}
D_{j}^{n,-} \leq&\; -\left[\left(\frac{\Delta t}{\Delta x}|a^-| + 2\frac{\Delta t}{{\Delta x}^2}\epsilon^- \right) - \left(\frac{\Delta t}{\Delta x}|a^-| + \frac{\Delta t}{{\Delta x}^2}\epsilon^- \right)e^{-\mu {\Delta x}} - \left(\frac{\Delta t}{{\Delta x}^2}\epsilon^- \right)e^{\mu {\Delta x}}\right]\frac{\Delta x}{\Delta t} \sum_{j=1}^{J}\left(U_{j}^{n,-}\right)^2e^{\mu x_j}\\
&\;-\left(\frac{\Delta t}{\Delta x}|a^-| + \frac{\Delta t}{{\Delta x}^2}\epsilon^- \right)\left(1 - \frac{\Delta t}{\Delta x}|a^-| - 2\frac{\Delta t}{{\Delta x}^2}\epsilon^- \right)\frac{\Delta x}{\Delta t} \sum_{j=1}^{J}\left(U_{j+1}^{n,-} - U_{j}^{n,-}\right)^2e^{\mu x_j}\\
&\;-e^{\mu {\Delta x}}\left(\frac{\Delta t}{{\Delta x}^2}\epsilon^- \right)\left(1 - \frac{\Delta t}{\Delta x}|a^-| - 2\frac{\Delta t}{{\Delta x}^2}\epsilon^- \right)\frac{\Delta x}{\Delta t} \sum_{j=1}^{J}\left(U_{j+1}^{n,-} - U_{j}^{n,-}\right)^2e^{\mu x_j}\\
&\;+\left(\frac{\Delta t}{\Delta x}|a^-|e^{-\mu {\Delta x}}  + \left(e^{-\mu {\Delta x}} -1\right)\frac{\Delta t}{{\Delta x}^2}\epsilon^- \right)\frac{\Delta x}{\Delta t} \left[ \left(U_{J+1}^{n,-}\right)^2e^{\mu x_{J+1}} - \left(U_{1}^{n,-}\right)^2e^{\mu x_1}\right]\\
&\; + 2\left(\frac{\Delta t}{{\Delta x}^2}\epsilon^- \right)\frac{\Delta x}{\Delta t} \left[\left(U_{J+1}^{n,-}\right)\left(U_{J+1}^{n,-} - U_{J}^{n,-}\right)e^{\mu x_{J+1}} - \left(U_{1}^{n,-}\right)\left(U_{1}^{n,-} - U_{0}^{n,-}\right)e^{\mu x_1} \right]\\
&\;-\left(\frac{\Delta t}{{\Delta x}^2}\epsilon^- \right)\left(\frac{\Delta t}{\Delta x}|a^-| + 2\frac{\Delta t}{{\Delta x}^2}\epsilon^- \right)\frac{\Delta x}{\Delta t} \left[\left(U_{J+1}^{n,-} - U_{J}^{n,-}\right)^2e^{\mu x_{J+1}} - \left(U_{1}^{n,-} - U_{0}^{n,-}\right)^2e^{\mu x_1}\right].
\end{split}
\end{equation}

From \eqref{eq:AppxDtLyafun006}, \eqref{eq:AppxDtLyafun007}, for $ n = 0,1,\ldots $, the discrete time derivative \eqref{eq:AppxDtLyafun01} is estimated as
\begin{equation}\label{eq:AppxDtLyafun008}
\begin{split}
\frac{\mathcal{L}^{n+1} - \mathcal{L}^{n}}{\Delta t} \leq&\; {\Delta x}  \sum_{j=1}^{J} U_{j}^{n \top} Q_j(\mu) U_{j}^{n}\\
&\;+ \begin{pmatrix} U_{0}^{n,+}  \\ U_{J+1}^{n,-} \end{pmatrix}^\top \left(\begin{smallmatrix}\left(a^+ + \frac{\epsilon^+}{{\Delta x}}\left(e^{-\mu {\Delta x}} -1\right)e^{\mu {\Delta x}} \right)e^{-\mu x_{1}} & 0 \\ 0 &  \left(|a^-| + \frac{\epsilon^-}{{\Delta x}}\left(e^{-\mu {\Delta x}} -1\right)e^{\mu {\Delta x}} \right)e^{\mu x_{J}} \end{smallmatrix} \right)\begin{pmatrix} U_{0}^{n,+}  \\ U_{J+1}^{n,-} \end{pmatrix}\\
&\;- \begin{pmatrix} U_{J}^{n,+}  \\ U_{1}^{n,-} \end{pmatrix}^\top \left(\begin{smallmatrix}\left(a^+ + \frac{\epsilon^+}{{\Delta x}}\left(e^{-\mu {\Delta x}} -1\right)e^{\mu {\Delta x}} \right)e^{-\mu x_{J+1}} & 0 \\ 0 &  \left(|a^-| + \frac{\epsilon^-}{{\Delta x}}\left(e^{-\mu {\Delta x}} -1\right)e^{\mu {\Delta x}} \right)e^{\mu x_{0}} \end{smallmatrix} \right)\begin{pmatrix} U_{J}^{n,+}  \\ U_{1}^{n,-} \end{pmatrix}\\
&\;-2 \begin{pmatrix} U_{0}^{n,+}  \\ U_{J+1}^{n,-} \end{pmatrix}^\top \left(\begin{smallmatrix} \frac{\epsilon^+}{{\Delta x}}e^{-\mu x_{0}} & 0 \\ 0 &  -\frac{\epsilon^-}{{\Delta x}}e^{\mu x_{J+1}} \end{smallmatrix} \right)\begin{pmatrix} U_{1}^{n,+} - U_{0}^{n,+} \\ U_{J+1}^{n,-} - U_{J}^{n,-}\end{pmatrix}\\
&\;+2 \begin{pmatrix} U_{J}^{n,+}  \\ U_{1}^{n,-} \end{pmatrix}^\top \left(\begin{smallmatrix} \frac{\epsilon^+}{{\Delta x}}e^{-\mu x_{J}} & 0 \\ 0 &  -\frac{\epsilon^-}{{\Delta x}}e^{\mu x_{1}} \end{smallmatrix} \right)\begin{pmatrix} U_{J+1}^{n,+} - U_{J}^{n,+} \\ U_{1}^{n,-} - U_{0}^{n,-}\end{pmatrix}\\
&\;-\begin{pmatrix} U_{1}^{n,+} - U_{0}^{n,+} \\ U_{J+1}^{n,-} - U_{J}^{n,-} \end{pmatrix}^\top \left(\begin{smallmatrix} \frac{\epsilon^+}{{\Delta x}}\left(a^+ + 2\frac{\epsilon^+}{\Delta x} \right) e^{-\mu x_{0}} & 0 \\ 0 &  \frac{\epsilon^-}{{\Delta x}}\left(|a^-| + 2\frac{\epsilon^-}{\Delta x} \right)e^{\mu x_{J+1}} \end{smallmatrix} \right)\begin{pmatrix} U_{1}^{n,+} - U_{0}^{n,+} \\ U_{J+1}^{n,-} - U_{J}^{n,-} \end{pmatrix}\\
&\;+\begin{pmatrix} U_{J+1}^{n,+} - U_{J}^{n,+} \\ U_{1}^{n,-} - U_{0}^{n,-} \end{pmatrix}^\top \left(\begin{smallmatrix} \frac{\epsilon^+}{{\Delta x}}\left(a^+ + 2\frac{\epsilon^+}{\Delta x} \right) e^{-\mu x_{J}} & 0 \\ 0 &  \frac{\epsilon^-}{{\Delta x}}\left(|a^-| + 2\frac{\epsilon^-}{\Delta x} \right)e^{\mu x_{1}} \end{smallmatrix} \right)\begin{pmatrix}U_{J+1}^{n,+} - U_{J}^{n,+} \\ U_{1}^{n,-} - U_{0}^{n,-} \end{pmatrix}
\end{split}
\end{equation}
where 
\begin{equation*}
Q_j(\mu) := \text{diag}\{q^+(\mu),\; q^-(\mu) \}P_j(\mu),\quad j = 1, \ldots, J,
\end{equation*}
with 
\begin{equation*}
\begin{split}
q^+(\mu):=&\;-\frac{1}{\Delta t}\left[\left(\frac{\Delta t}{\Delta x}a^+ + 2\frac{\Delta t}{{\Delta x}^2}\epsilon^+ \right) -  \left(\frac{\Delta t}{\Delta x}a^+ + \frac{\Delta t}{{\Delta x}^2}\epsilon^+ \right)e^{-\mu {\Delta x}} - \left(\frac{\Delta t}{{\Delta x}^2}\epsilon^+ \right)e^{\mu {\Delta x}}\right], \\
=&\;\left(e^{-\mu {\Delta x}} - 1\right)\left[\frac{a^+}{\Delta x}  - \left(\frac{\epsilon^+}{{\Delta x}^2} \right)\left(e^{\mu {\Delta x}} - 1\right)\right],\\
q^-(\mu):=&\; -\frac{1}{\Delta t}\left[\left(\frac{\Delta t}{\Delta x}|a^-| + 2\frac{\Delta t}{{\Delta x}^2}\epsilon^- \right) -  \left(\frac{\Delta t}{\Delta x}|a^-| + \frac{\Delta t}{{\Delta x}^2}\epsilon^- \right)e^{-\mu {\Delta x}} - \left(\frac{\Delta t}{{\Delta x}^2}\epsilon^- \right)e^{\mu {\Delta x}}\right],\\
=&\;\left(e^{-\mu {\Delta x}} - 1\right)\left[\frac{|a^-|}{\Delta x}  - \left(\frac{\epsilon^-}{{\Delta x}^2} \right)\left(e^{\mu {\Delta x}} - 1\right)\right]. 
\end{split}
\end{equation*}

We use the boundary conditions \eqref{eq:discModifiedlinsystemBCs} in the inequality \eqref{eq:AppxDtLyafun008} to obtain  
\begin{equation}\label{eq:AppxDtLyafun09}
\begin{split}
\frac{\mathcal{L}^{n+1} - \mathcal{L}^{n}}{\Delta t} \leq &\; {\Delta x}  \sum_{j=1}^{J} U_{j}^{n \top} Q_j(\mu) U_{j}^{n}\\
&\begin{aligned}
+ \begin{pmatrix} U_{J}^{n,+}  \\ U_{1}^{n,-} \end{pmatrix}^\top & \left(K^\top \left(\begin{smallmatrix}\left(a^+ + \frac{\epsilon^+}{{\Delta x}}\left(e^{-\mu {\Delta x}} -1\right)e^{\mu {\Delta x}} \right)e^{-\mu x_{1}} & 0 \\ 0 &  \left(|a^-| + \frac{\epsilon^-}{{\Delta x}}\left(e^{-\mu {\Delta x}} -1\right)e^{\mu {\Delta x}} \right)e^{\mu x_{J}} \end{smallmatrix} \right)K \right. \\
&\left. - \left(\begin{smallmatrix}\left(a^+ + \frac{\epsilon^+}{{\Delta x}}\left(e^{-\mu {\Delta x}} -1\right)e^{\mu {\Delta x}} \right)e^{-\mu x_{J+1}} & 0 \\ 0 &  \left(|a^-| + \frac{\epsilon^-}{{\Delta x}}\left(e^{-\mu {\Delta x}} -1\right)e^{\mu {\Delta x}} \right)e^{\mu x_{0}} \end{smallmatrix} \right) \right)\begin{pmatrix} U_{J}^{n,+}  \\ U_{1}^{n,-} \end{pmatrix}
\end{aligned}\\
&- \begin{pmatrix} U_{J}^{n,+}  \\ U_{1}^{n,-} \end{pmatrix}^\top \left(K^\top \left(\begin{smallmatrix} \frac{\epsilon^+}{{\Delta x}}e^{-\mu x_{0}} & 0 \\ 0 &  -\frac{\epsilon^-}{{\Delta x}}e^{\mu x_{J+1}} \end{smallmatrix} \right) K   - \left(\begin{smallmatrix} \frac{\epsilon^+}{{\Delta x}}e^{-\mu x_{J}} & 0 \\ 0 &  -\frac{\epsilon^-}{{\Delta x}}e^{\mu x_{1}} \end{smallmatrix} \right)\right) \begin{pmatrix} U_{J+1}^{n,+} - U_{J}^{n,+} \\ U_{1}^{n,-} - U_{0}^{n,-} \end{pmatrix}\\
&\begin{aligned}
- \begin{pmatrix} U_{J+1}^{n,+} - U_{J}^{n,+} \\ U_{1}^{n,-} - U_{0}^{n,-} \end{pmatrix}^\top & \left(K^\top \left(\begin{smallmatrix} \frac{\epsilon^+}{{\Delta x}}\left(a^+ + 2\frac{\epsilon^+}{\Delta x} \right) e^{-\mu x_{0}} & 0 \\ 0 &  \frac{\epsilon^-}{{\Delta x}}\left(|a^-| + 2\frac{\epsilon^-}{\Delta x} \right)e^{\mu x_{J+1}} \end{smallmatrix} \right)K \right. \\
&\left. - \left(\begin{smallmatrix} \frac{\epsilon^+}{{\Delta x}}\left(a^+ + 2\frac{\epsilon^+}{\Delta x} \right) e^{-\mu x_{J}} & 0 \\ 0 &  \frac{\epsilon^-}{{\Delta x}}\left(|a^-| + 2\frac{\epsilon^-}{\Delta x} \right)e^{\mu x_{1}} \end{smallmatrix} \right) \right)\begin{pmatrix} U_{J+1}^{n,+} - U_{J}^{n,+} \\ U_{1}^{n,-} - U_{0}^{n,-} \end{pmatrix}
\end{aligned},\; n = 0,1,\ldots.
\end{split}
\end{equation}

By \eqref{eq:NumThm01}, \eqref{eq:NumThm02} and \eqref{eq:NumThm03}, we have 
\begin{equation}\label{eq:AppxDtLyafun10}
\frac{\mathcal{L}^{n+1} - \mathcal{L}^{n}}{\Delta t} \leq {\Delta x} \sum_{j=1}^{J} U_{j}^{n \top} Q_j(\mu) U_{j}^{n},\quad n = 0,1,\ldots.
\end{equation}

For $ \mu {\Delta x} \geq 0 $, using Taylor's expansion, we estimate 
\begin{equation}\label{eq:AppxDtLyafun11}
e^{-\mu {\Delta x}} - 1 \leq -\mu {\Delta x}e^{-\mu {\Delta x}},\quad \text{and} \quad e^{\mu {\Delta x}} - 1 \leq \mu {\Delta x}e^{\mu {\Delta x}}.
\end{equation} 
From \eqref{eq:AppxDtLyafun11}, we have 
\begin{equation}\label{eq:AppxDtLyafun12}
\begin{split}
&\frac{a^+}{\Delta x}  - \left(\frac{\epsilon^+}{{\Delta x}^2} \right)\left(e^{\mu {\Delta x}} - 1\right) \geq \frac{1}{\Delta x} \left(a^+  - \epsilon^+\mu e^{\mu {\Delta x}}\right),\quad \text{and}\\
&\frac{|a^-|}{\Delta x}  - \left(\frac{\epsilon^-}{{\Delta x}^2} \right)\left(e^{\mu {\Delta x}} - 1\right) \geq \frac{1}{\Delta x} \left(|a^-|  - \epsilon^-\mu e^{\mu {\Delta x}}\right).   
\end{split}
\end{equation}
Further, from \eqref{eq:AppxDtLyafun11} and \eqref{eq:AppxDtLyafun12}, if $$ \mu e^{\mu {\Delta x}} \leq \frac{\alpha}{\epsilon}, $$
then
\begin{equation}\label{eq:AppxDtLyafun13}
Q_j(\mu) + \left( \alpha\mu e^{-\mu {\Delta x}} - \epsilon\mu^2 \right)P_j(\mu),
\end{equation}
is negative definite for all $j = 1,\ldots,J$. Therefore, for $ \mu > 0 $ such that $ \mu e^{\mu {\Delta x}} \leq \frac{\alpha}{\epsilon} $, we have 
\begin{equation}\label{eq:AppxDtLyafun14}
\frac{\mathcal{L}^{n+1} - \mathcal{L}^{n}}{\Delta t} \leq -\eta_N \mathcal{L}^{n},\quad n = 0,1,\ldots,
\end{equation}
where \begin{equation}\label{eqn:eta_N} \eta_N := \alpha\mu e^{-\mu {\Delta x}} - \epsilon\mu^2.\end{equation} Hence, 
\begin{equation}
\mathcal{L}^{n+1} \leq \left(1 - {\Delta t}\eta_N\right)\mathcal{L}^{n} \leq \left(1 - {\Delta t}\eta_N\right)^{n+1}\mathcal{L}^{0} \leq e^{-\eta_N t^{n+1}}\mathcal{L}^{0},\quad n = 0,1,\ldots. \label{eq:AppxDtLyafun-system-08}  
\end{equation}
Thus it can be concluded that by the discrete norm equivalence \eqref{eq:DiscNormEquivalence}, we can obtain the condition for exponential stability \eqref{eq:DiscreteExpStability}. 
\end{proof}

Having proved the main result above, in the next section we present some numerical results which will demonstrate the theoretical results given above.

\section{Numerical computations}\label{sec3}
In this section, we illustrate Theorem \ref{NumericalStabThem} and show the influence of artificial viscosity on the decay of the solution of $ 2\times 2 $ linear systems of conservation laws in one dimension. In particular, we study this for the linear wave equation in Section \ref{sec:linearwave}, the isothermal Euler's equations in Section \ref{sec:isothermalEuler} and the Saint-Venant equations in Section \ref{sec:SaintVenant}. In the examples below, we employ a matrix 
\begin{equation*}
K= \begin{pmatrix} e^{-\mu/2} & 0 \\ 0 & e^{-\mu/2} \end{pmatrix}, \quad \mu > 0
\end{equation*}
which satisfies the conditions \eqref{eq:NumThm01}, \eqref{eq:NumThm02}, \eqref{eq:NumThm03} given in Theorem \ref{NumericalStabThem}. 

Most importantly, it can be observed that as $ {\Delta x} \rightarrow 0 $, for the CFL condition $ 0 <  \frac{\Delta t}{\Delta x}\max\{ a^+,\; |a^-|\} \leq 1$ and for $ \mu > 0 $, $ \epsilon \rightarrow 0 $ and the decay rates, $ \eta_T $ in Equation \eqref{eqn:eta_T} and $ \eta_N $ in Equation \eqref{eqn:eta_N} converge to the expected decay rate, $ \alpha \mu $. Further as $ \mu $ is increased, the decay rate also increases for the given matrix $ K $. Note that in the computations below, we use a stopping criterion with a tolerance of $ 10^{-07} $ for time $ T $. 

\subsection{Example 1 - linear wave equation}\label{sec:linearwave}
The wave equation
\begin{equation}\label{WaveEqn}
\partial_{tt} w(t,x)  =c^2 \partial_{xx} w(t,x), \quad (t,x)\in [0,\infty)\times [0,1],
\end{equation}
where $ w(x,t)$ is the wave function, and $ c = 1 $ the characteristic phase velocity, can be written as system \eqref{linSystem} with $ a^+ = 1 $, $ a^- = -1 $, and the Riemann invariants can be obtained by
\eqref{RiemannCoordinates} as 
\begin{equation*}
U^+:= -\partial_t w + \partial_x w, \quad U^-:= \partial_t w + \partial_x w. 
\end{equation*} 

Since $ a^+ = |a^-| = 1 $, $ \alpha = \min\{ a^+,\; |a^-|\} = 1 $, the theoretical and numerical decay rates are given by $\eta_T = \mu - \epsilon\mu^2 $, see Equation \eqref{eqn:eta_T}, and $\eta_N = \mu e^{-\mu {\Delta x}} - \epsilon\mu^2 $, see Equation \eqref{eqn:eta_N}, with the diffusion coefficient $ \epsilon = {\frac{1}{2} \dx\Big(1 - \frac{\dt}{\dx}\Big)} $  and the CFL condition $ 0 < {\frac{\dt}{\dx}} \leq 1 $.  The discrete system \eqref{eq:discModified} approximates \eqref{WaveEqn}, and the discrete Lyapunov function is bounded above by 
\begin{equation}\label{UpperBound}
\mathcal{L}^{n}_{\text{up}} = e^{-\eta_N t^n}\mathcal{L}^{0}, \; n = 0,1, \ldots. 
\end{equation} 

We first take a constant initial data $ U_j^{0,+} = -0.5$, $U_j^{0,-} = 0.5,\; j = 1, \ldots,J $. For the choice of the CFL condition $ 0.95 $ or $ 0.5 $, $ \mu = 0.5 $, maximum time run $ T = 12 $ and the given matrix $ K $,  the $ L_{\Delta x}^\infty $  and $ L_{\Delta x}^2 $ difference between the upper bound $ \mathcal{L}_{\text{up}}^{n} $ and $ \mathcal{L}^{n} $, and the rate of convergence of the discrete Lyapunov function $ \mathcal{L}^{n} $ for different number of grid points are computed and summarised in Table \ref{table:Convergence-01} and \ref{table:Convergence-02}.

\begin{table}[H]
	\centering
	\begin{tabular}{ccccccc}
		\hline \\[0.5ex]
		$ J $ & $ \|\mathcal{L}_{\text{up}} - \mathcal{L}^n \|_{L_{\Delta x}^\infty} $ & $ \|\mathcal{L}_{\text{up}} - \mathcal{L}^n \|_{L_{\Delta x}^2} $ & $ \alpha \mu $  & $ \eta_T $ & $ \eta_N $ &  Rate$ (\mathcal{L}^n) $ \\ [0.5ex]
		\hline \\[0.5ex]
		100  & 0.0025	  &	0.0030		 & 0.5	 &	0.4999	& 0.4974 & 2.0298\\[1ex] 	
		200  & 0.0013 	  & 0.0016       & 0.5   &  0.5000  & 0.4987 & 2.0391\\[1ex]
		400  & 7.2780e-04 & 8.9220e-04   & 0.5   &  0.5000  & 0.4994 & 2.0545\\[1ex]
		800  & 4.0501e-04 & 5.0287e-04   & 0.5   &  0.5000  & 0.4997 & 2.0785\\[1ex]
		1600 & 2.2993e-04 & 2.9221e-04   & 0.5   &  0.5000  & 0.4998 & 2.1153\\[0.5ex]
		\hline \\
	\end{tabular} 
	\caption[]%
	{The comparison of the upper bound of Lyapunov function with discrete Lyapunov function. Under $ {\Delta x} = \frac{1}{J} $, $ {\Delta t} = \text{CFL} {\Delta x}$, CFL = 0.95, $ \mu = 0.5 $, and $ T = 12 $. Rate$ (\mathcal{L}^n) $ = $ \|\mathcal{L}_{J}^n - \mathcal{L}_{2J}^n \|_{L_{\Delta x}^2}/\|\mathcal{L}_{2J}^n - \mathcal{L}_{4J}^n \|_{L_{\Delta x}^2} $.}
	\label{table:Convergence-01}
\end{table} 

\begin{table}[H]
	\centering
	\begin{tabular}{ccccccc}
		\hline \\[0.5ex]
		$ J $ & $ \|\mathcal{L}_{\text{up}} - \mathcal{L}^n \|_{L_{\Delta x}^\infty} $ & $ \|\mathcal{L}_{\text{up}} - \mathcal{L}^n \|_{L_{\Delta x}^2} $ & $ \alpha \mu $  & $ \eta_T $ & $ \eta_N $ &  Rate$ (\mathcal{L}^n) $ \\ [0.5ex]
		\hline \\[0.5ex]
		100  & 0.0027	  &	0.0052		& 0.5	   &  0.4994  & 0.4969 & 2.0874\\[1ex] 	
		200  & 0.0016     & 0.0031      & 0.5      &  0.4997  & 0.4984 & 2.1270\\[1ex]
		400  & 0.0010     & 0.0019      & 0.5      &  0.4998  & 0.4992 & 2.1922\\[1ex]
		800  & 6.2921e-04 & 0.0012      & 0.5      &  0.4999  & 0.4996 & 2.3004\\[1ex]
		1600 & 4.0021e-04 & 7.9843e-04  & 0.5      &  0.5000  & 0.4998 & 2.4773\\[0.5ex]
		\hline \\
	\end{tabular} 
	\caption[]%
	{The comparison of the upper bound of Lyapunov function with discrete Lyapunov function. Under $ {\Delta x} = \frac{1}{J} $, $ {\Delta t} = \text{CFL} {\Delta x}$, CFL = 0.5, $ \mu = 0.5 $, and $ T = 12 $. Rate$ (\mathcal{L}^n) $ = $ \|\mathcal{L}_{J}^n - \mathcal{L}_{2J}^n \|_{L_{\Delta x}^2}/\|\mathcal{L}_{2J}^n - \mathcal{L}_{4J}^n \|_{L_{\Delta x}^2} $.}
	\label{table:Convergence-02}
\end{table} 

In the above two tables, one observes that with the grid refinement the discrepancy from the upper bound in $\mathcal{L}$ is reduced.  The magnitude of the difference also depends on the CFL condition, the larger the CFL number the smaller are the errors. In addition both $\eta_T$ and $\eta_N$ approach $\alpha \mu$ as expected. One also observes that the expected rate of convergence of the discrete Lyapunov function $ \mathcal{L}^{n} $ is guaranteed. 

For a fixed number of grid points, $ J = 1600 $, different choices of $ \mu > 0 $ and CFL = 0.95, we computed the $ L_{\Delta x}^\infty $  and $ L_{\Delta x}^2 $ difference between the upper bound $ \mathcal{L}_{\text{up}}^{n} $ and $ \mathcal{L}^{n} $  for above constant initial data and for a perturbation of the initial data, $ U_j^{0,+} = -0.5 + \frac{1}{4\pi }\sin\left(2\pi x_j\right)$, $U_j^{0,-} = 0.5 + \frac{1}{4\pi }\sin\left(2\pi x_j\right),\; j = 1, \ldots,J $, respectively. The result is summarized in Table \ref{table:Convergence-03} and  \ref{table:Convergence-04}.  
\begin{table}[H]
	\centering
	\begin{tabular}{cccccc}
		\hline \\[0.5ex]
		$ \mu $ & $ \|\mathcal{L}_{\text{up}} - \mathcal{L}^n \|_{L_{\Delta x}^\infty} $ & $ \|\mathcal{L}_{\text{up}} - \mathcal{L}^n \|_{L_{\Delta x}^2} $ & $ \alpha \mu $  & $ \eta_T $ & $ \eta_N $  \\ [0.5ex]
		\hline \\[0.5ex]
		0.25  & 1.0106e-04	  &	1.6670e-04   & 0.25	 &	0.2500	& 0.2500  \\[1ex] 	
		0.5   & 2.2993e-04 	  & 2.9221e-04   & 0.5   &  0.5000  & 0.4998  \\[1ex]
		1.25  & 7.7158e-04    & 7.8381e-04   & 1.25  &  1.2500  & 1.2490  \\[1ex]
		2.75  & 0.0048        & 0.0034       & 2.7   &  2.7499  & 2.7452  \\[1ex]
		4.5   & 0.0291        & 0.0165       & 4.5   &  4.4997  & 4.4870  \\[0.5ex]
		\hline \\
	\end{tabular} 
	\caption[]%
	{The comparison of the upper bound of Lyapunov function with discrete Lyapunov function for constant initial data. Under $ {\Delta x} = \frac{1}{1600} $, $ {\Delta t} = \text{CFL} {\Delta x}$, CFL = 0.95, and $ T = 35 $.}
	\label{table:Convergence-03}
\end{table} 

\begin{table}[H]
	\centering
	\begin{tabular}{cccccc}
		\hline \\[0.5ex]
		$ \mu $ & $ \|\mathcal{L}_{\text{up}} - \mathcal{L}^n \|_{L_{\Delta x}^\infty} $ & $ \|\mathcal{L}_{\text{up}} - \mathcal{L}^n \|_{L_{\Delta x}^2} $ & $ \alpha \mu $  & $ \eta_T $ & $ \eta_N $  \\ [0.5ex]
		\hline \\[0.5ex]
		0.25  & 1.7610e-04	  &	2.4173e-04   & 0.25	 &	0.2500	& 0.2500  \\[1ex] 	
		0.5   & 2.3129e-04 	  & 2.8947e-04   & 0.5   &  0.5000  & 0.4998  \\[1ex]
		1.25  & 6.6414e-04    & 7.4674e-04   & 1.25  &  1.2500  & 1.2490  \\[1ex]
		2.75  & 0.0046        & 0.0033       & 2.7   &  2.7499  & 2.7452  \\[1ex]
		4.5   & 0.0284        & 0.0159       & 4.5   &  4.4997  & 4.4870  \\[0.5ex]
		\hline \\
	\end{tabular} 
	\caption[]%
	{The comparison of the upper bound of Lyapunov function with discrete Lyapunov function for perturbation of the initial data. Under $ {\Delta x} = \frac{1}{1600} $, $ {\Delta t} = \text{CFL} {\Delta x}$, CFL = 0.95, and $ T = 35 $.}
	\label{table:Convergence-04}
\end{table} 

In Table \ref{table:Convergence-03} and \ref{table:Convergence-04} it can be observed that for the constant initial value and the perturbed initial value, the discrete Lyapunov function is closer to the upper bound for the smaller $\mu$. This explains the fact that the larger values of $\mu$ imply that the matrix $K$ has stronger influence on the Lyapunov function. It can also be noted that the influence of $\mu$ in the decay rate is diminished by the presence of the term $e^{-\mu\Delta x}$ in the first term of $\eta_N$. This seems to be the case even-though the quadratic term would be dominant. Hence the discrete Lyapunov function remains farther from the upper bound for larger $\mu$.  In all cases $\eta_T$ and $\eta_N$ approach $\alpha \mu$. Thus larger values of $\mu$ dominate the discrete Lyapunov function which is also attributed to the choice of $K$.
 
In Figure \ref{fig:LyapunovProfile-01}, we compare different upper-bounds of the Lyapunov function:
\begin{itemize}
\item $\mathcal{L}_{\text{up}}(\alpha\mu)$ - the upper-bound of Equation \eqref{conservationlaw} or Equation \eqref{eq:LyapFunSystem} with $\epsilon = 0$;
\item $\mathcal{L}_{\text{up}}(\eta_T)$ - the upper-bound of the Lyapunov function in Equation \eqref{eq:LyapFunSystem} as established in Lemma \ref{lem:contsStab};
\item $\mathcal{L}_{\text{up}}(\eta_N)$ - the upper-bound of the discrete Lyapunov function in Equation \eqref{conservationlaw} as established in the proof of Theorem \ref{NumericalStabThem}.
\end{itemize}
It also compares the above with the $\mathcal{L}^n$ which is the discrete Lyapunov function stated in Equation \eqref{eq:dLyafun-system}.  It can also be seen that $\alpha\mu \ge \alpha\mu - \epsilon\mu^2 \ge \alpha\mu e^{-\mu \Delta{x}} - \epsilon \mu^2$ hence $\mathcal{L}_{\text{up}}(\alpha\mu) \le \mathcal{L}_{\text{up}}(\eta_T) \le \mathcal{L}_{\text{up}}(\eta_N)$. This is demonstrated in Figure \ref{fig:LyapunovProfile-01}. In the figure $\mathcal{L}^n$, has the fastest decay such that $\mathcal{L}^n \le \mathcal{L}_{\text{up}}(\alpha\mu)$. For fixed $\mu$ this can be attributed to the second-order terms in Equation \eqref{eq:discModifiedlinsystem}. This behaviour is consistent for both problems i.e. with constant or non-constant initial data. We also show that when the values of $ \mu > 0 $ increase the decay rate increases for constant and non-constant initial data in Figure \ref{fig:LyapunovProfile-02}. This can also be attributed to the boundary condition in which the matrix $K$ depends on $\mu$.

\begin{figure}[H]
	\centering
	\begin{subfigure}[b]{0.485\textwidth}
		\centering
		\includegraphics[width=\textwidth]{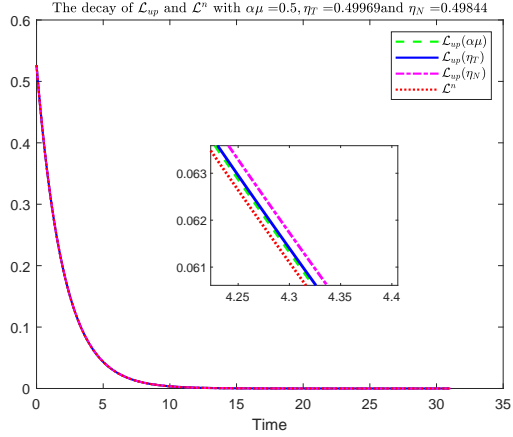}
		\caption[]%
		{ Constant initial data}    
		\label{fig:DecayRateComparion-01-01}
	\end{subfigure}
	\hfill
	\begin{subfigure}[b]{0.485\textwidth}  
		\centering 
		\includegraphics[width=\textwidth]{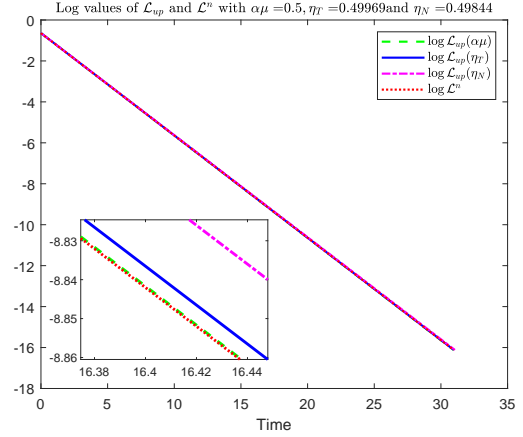}
		\caption[]%
		{ Constant initial data}    
		\label{fig:DecayRateComparion-01-02}
	\end{subfigure}
	\vskip\baselineskip
	\begin{subfigure}[b]{0.485\textwidth}   
		\centering 
		\includegraphics[width=\textwidth]{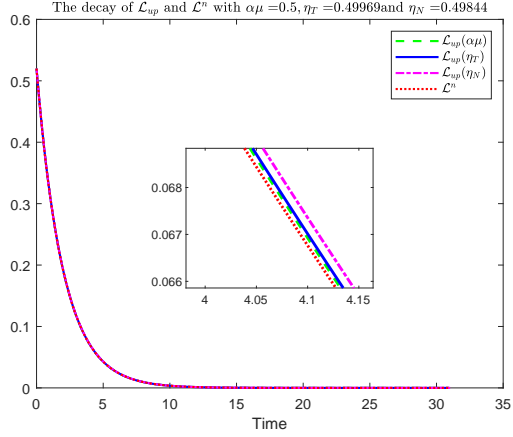}
		\caption[]%
		{Non-constant initial data}    
		\label{fig:DecayRateComparion-01-03}
	\end{subfigure}
	\hfill
	\begin{subfigure}[b]{0.485\textwidth}   
		\centering 
		\includegraphics[width=\textwidth]{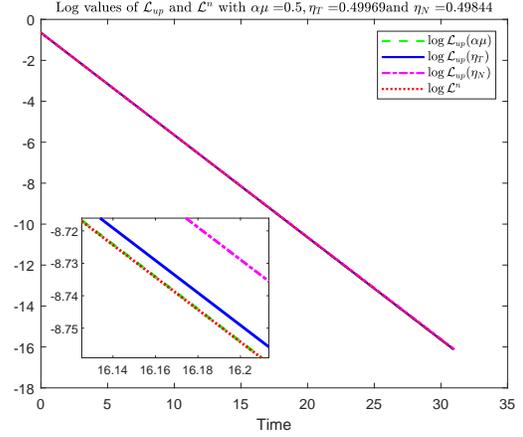}
		\caption[]%
		{Non-constant initial data}    
		\label{fig:DecayRateComparion-01-04}
	\end{subfigure}
	\caption[]%
	{ The comparison of the decay of Lyapunov function and the upper bound in (a) and (c), and its log values in (b) and (d) for constant and non-constant initial data. Under $ {\Delta x} = \frac{1}{200} $, $ {\Delta t} = \text{CFL}{\Delta x}$, CFL = 0.5, $ \mu = 0.5 $ and $ T = 35 $.} 
	\label{fig:LyapunovProfile-01}
\end{figure}

\begin{figure}[H]
	\centering
	\begin{subfigure}[b]{0.485\textwidth}
		\centering
		\includegraphics[width=\textwidth]{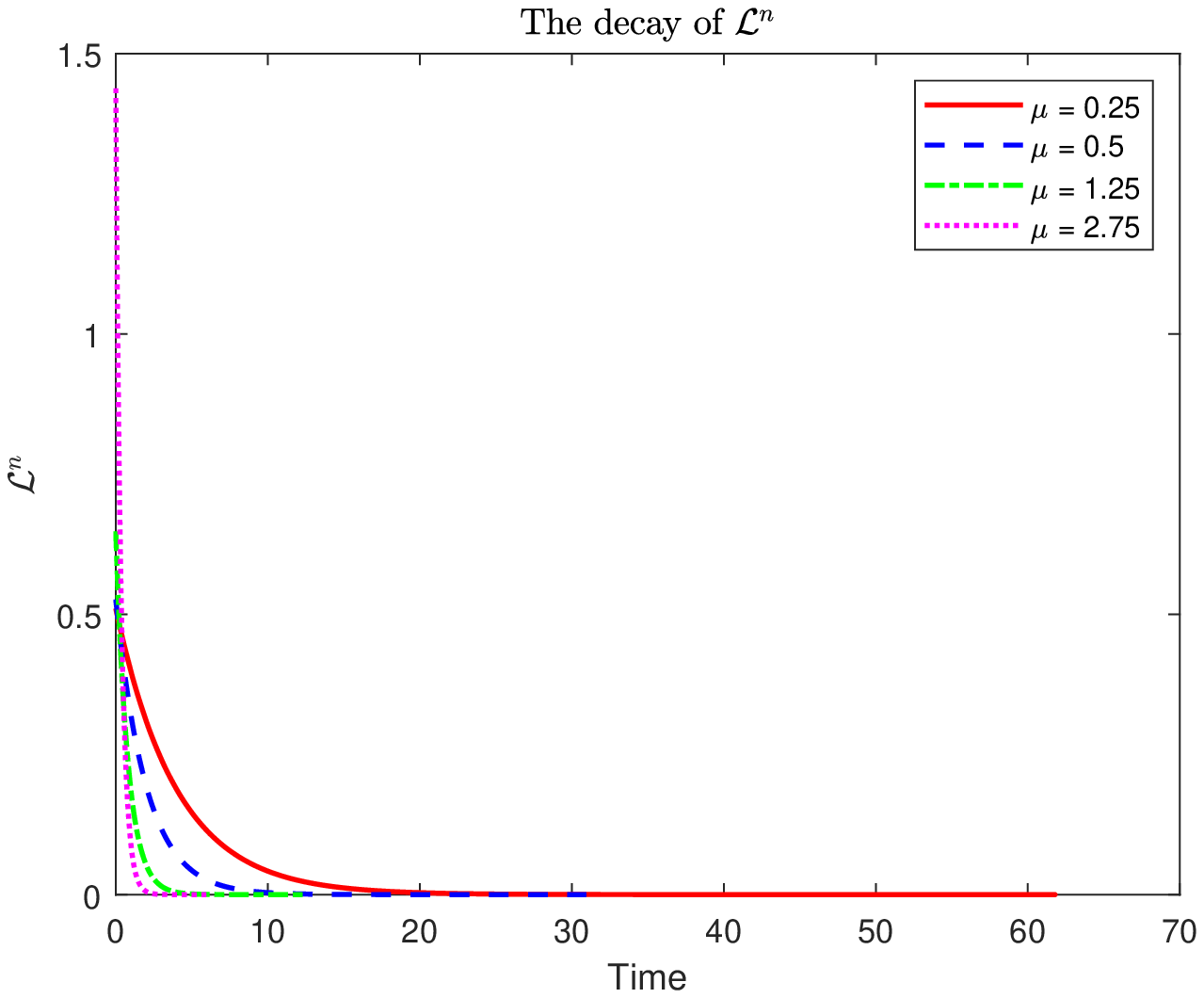}
		\caption[]%
		{ Constant initial data}    
		\label{fig:DecayRateComparion-02-01}
	\end{subfigure}
	\hfill
	\begin{subfigure}[b]{0.485\textwidth}  
		\centering 
		\includegraphics[width=\textwidth]{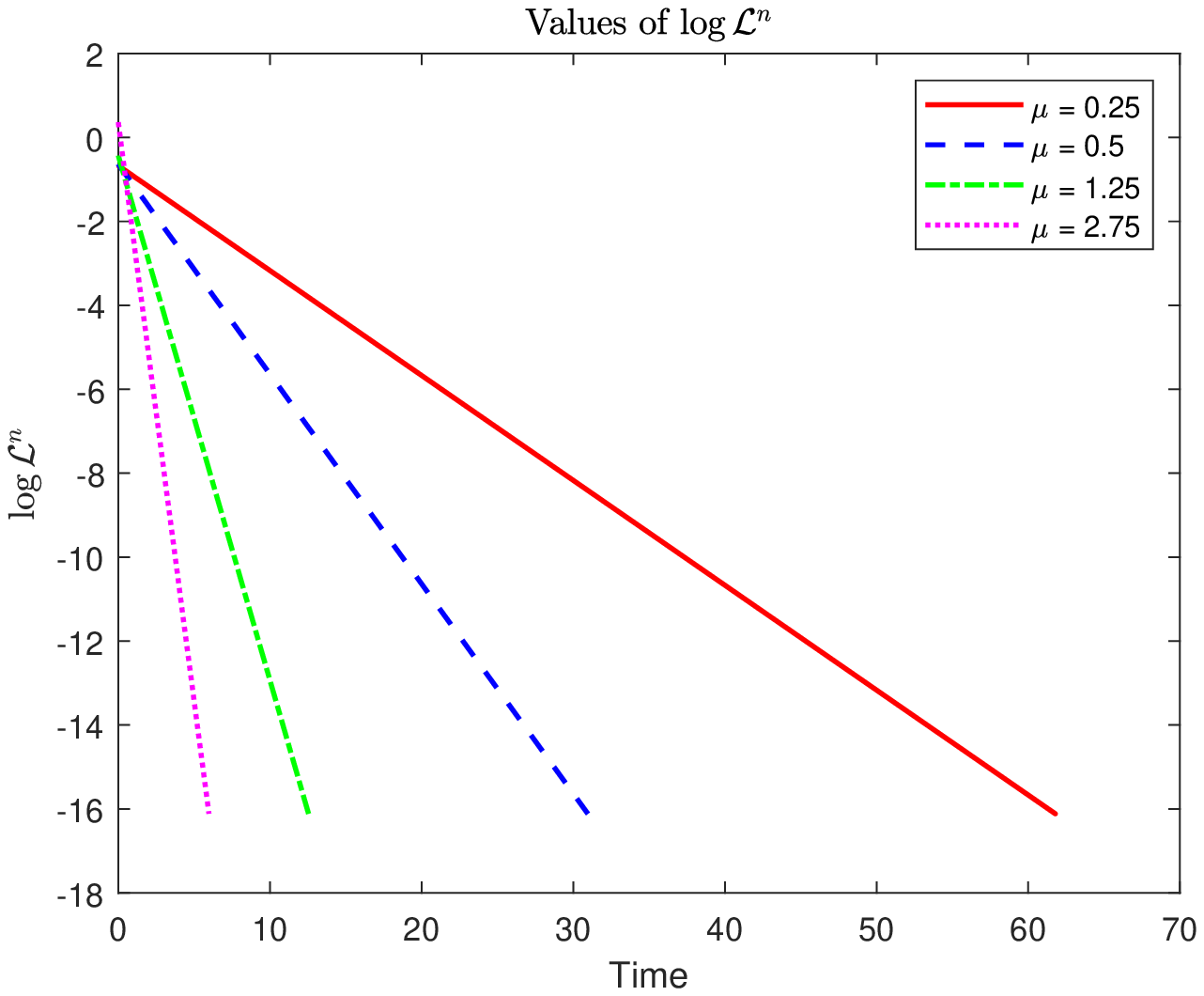}
		\caption[]%
		{ Constant initial data}    
		\label{fig:DecayRateComparion-02-02}
	\end{subfigure}
	\vskip\baselineskip
	\begin{subfigure}[b]{0.485\textwidth}   
		\centering 
		\includegraphics[width=\textwidth]{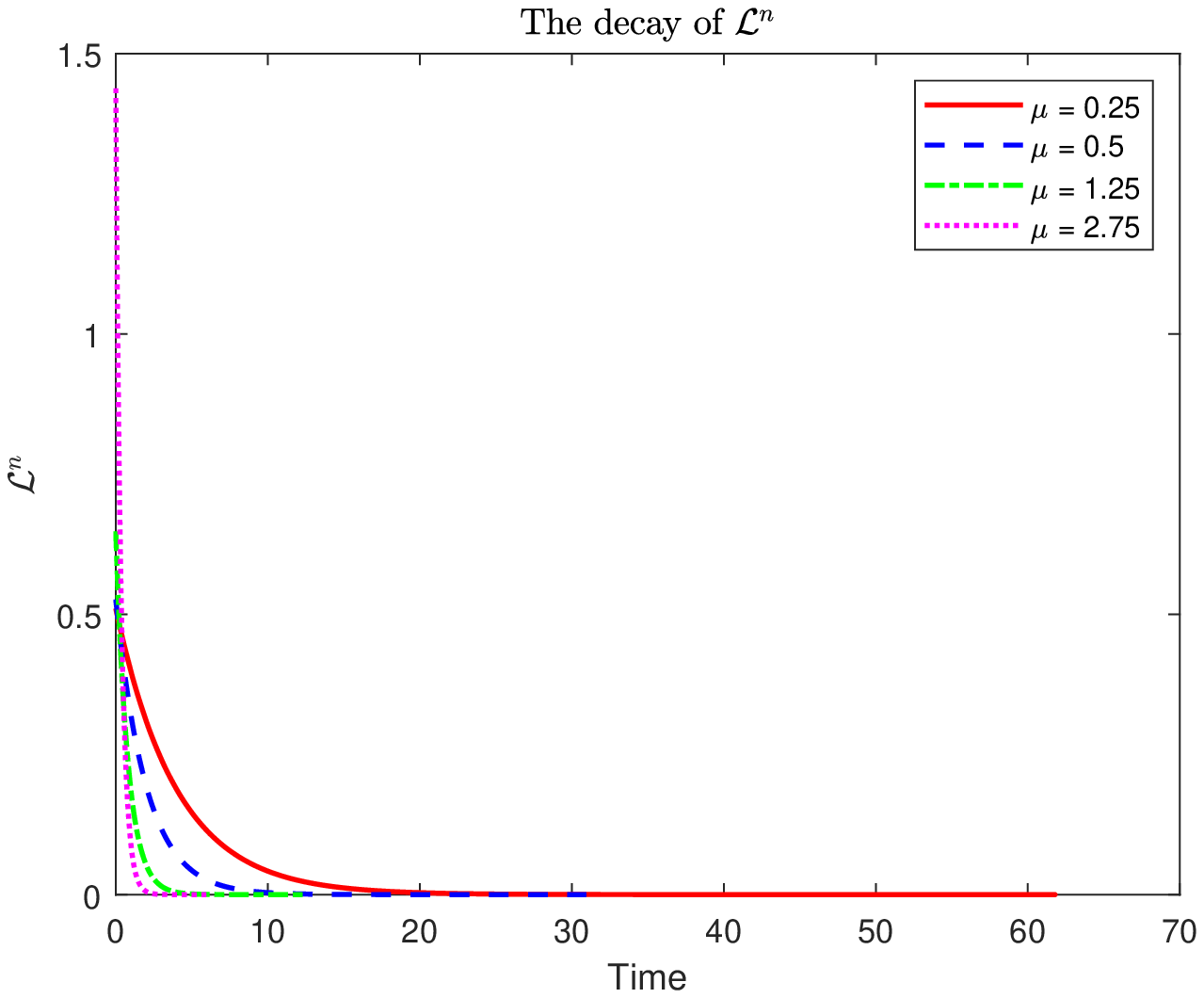}
		\caption[]%
		{Non-constant initial data}    
		\label{fig:DecayRateComparion-02-03}
	\end{subfigure}
	\hfill
	\begin{subfigure}[b]{0.485\textwidth}   
		\centering 
		\includegraphics[width=\textwidth]{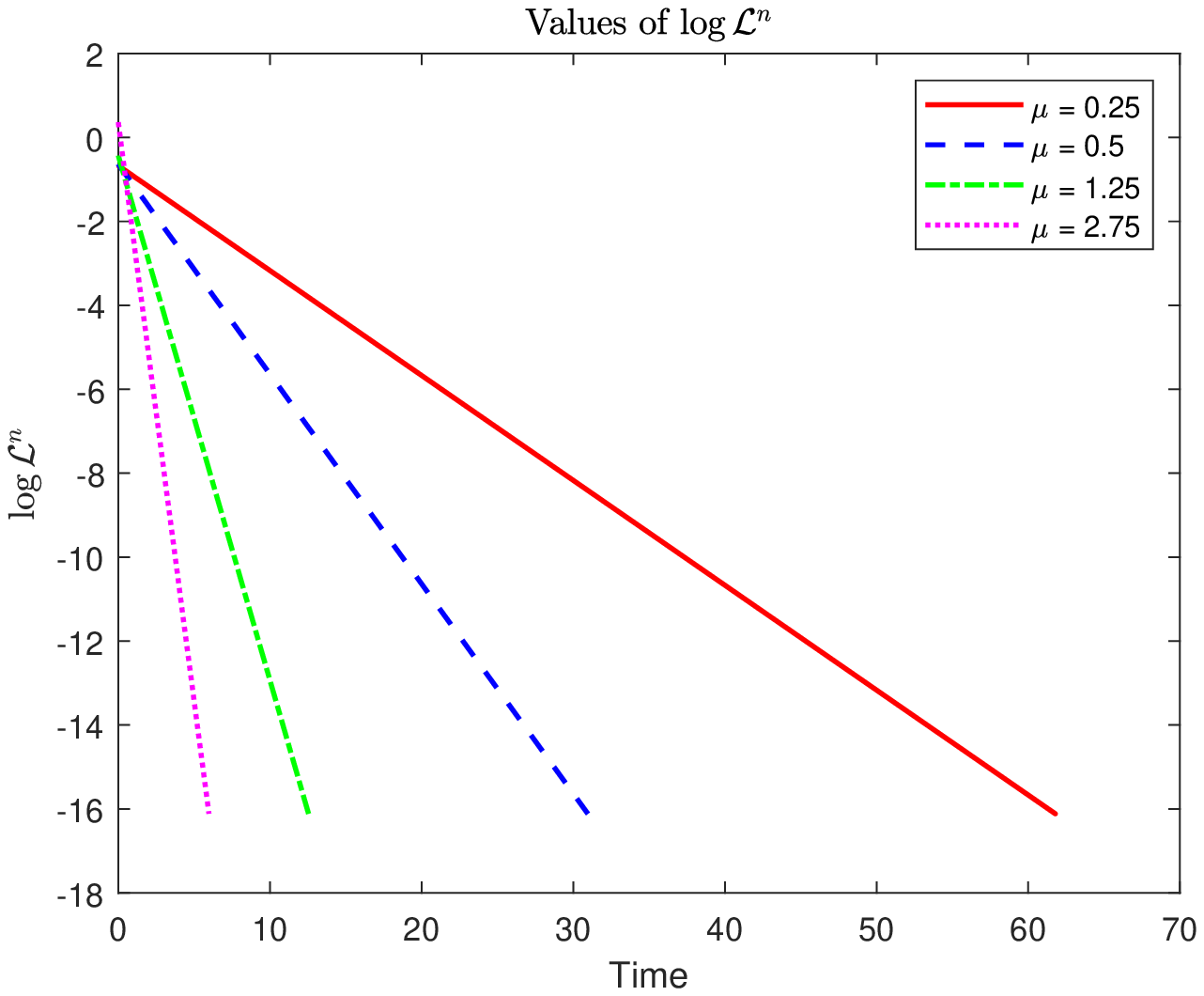}
		\caption[]%
		{Non-constant initial data}    
		\label{fig:DecayRateComparion-02-04}
	\end{subfigure}
	\caption[]%
	{ The comparison of the decay of Lyapunov function in (a) and (c), and its log values in (b) and (d) for different choices of $ \mu > 0 $ and for constant and non-constant initial data. Under $ {\Delta x} = \frac{1}{200} $, $ {\Delta t} = \text{CFL}{\Delta x}$, CFL = 0.5, and $ T = 70 $.} 
	\label{fig:LyapunovProfile-02}
\end{figure}
Of note is the fact that the upper bound due to $\eta_N$ has the slowest decay rate. This is caused by the slight influence of the $e^{-\mu\Delta{x}}$ term. The expected decay $\mathcal(\alpha\mu)$ i.e. $\epsilon = 0$, is almost similar to the theoretical decay with viscosity, $\mathcal{L}(\eta_T)$, i.e. $\text{CFL} < 1$. 

\subsection{Example 2: Isothermal Euler's equations} \label{sec:isothermalEuler}
The following Isothermal Euler Equations (see \cite{Banda_2006})
\begin{equation}\label{eq:Isothermal01}
\begin{split}
\partial_t \rho(t,x) + \partial_xq(t,x) =&\; 0,\\
\partial_tq(t,x) + \partial_x\left( \frac{q^2(t,x)}{\rho(t,x)} + a^2 \rho(t,x)\right) =& \; 0,
\end{split}
\end{equation}
where $ x \in [0,l] $, $ t \in [0, +\infty) $, $ \rho(t,x) $ is the density of a gas,  $ q(t,x):= \rho(t,x) u(t,x) $ is the mass flux, $ u(t,x) $ is the velocity of a gas and $ a $ is the speed of sound, can also be written as system \eqref{linSystem} with $ a^+ = q^*/\rho^* + a $, $ a^- = q^*/\rho^* - a $ (a sub-sonic flow is considered), and the Riemann invariants can be obtained by \eqref{RiemannCoordinates} as 
\begin{equation*}
U^+(t,x):= (q(t,x) - q^*) - a^-(\rho(t,x) - \rho^*), \quad U^-(t,x):= (q(t,x) - q^*) - a^+(\rho(t,x) - \rho^*),
\end{equation*}
where $ \rho^* $, $ q^* $ is a steady-state solution of the system \eqref{eq:Isothermal01} that satisfies 
\begin{equation}\label{eq:Isothermal02}
q^* = \kappa_1,\quad \frac{{q^*}^2}{\rho^*} + a^2 \rho^* = \kappa_2,
\end{equation}
where $ \kappa_1 $ and $ \kappa_2 $ are real constants. Similarly, the discrete systems \eqref{eq:discModified} approximates \eqref{eq:Isothermal01}. 

We take $ a = 1 $, $ q^* = 0.6 $ and $ \rho^* = 3 $. Thus, $ a^+ = 1.2 $ and $ a^- = -0.8 $. Then, we define an initial data by 
\begin{equation*}
 U_j^{0,+} = 0.8e^{-x_j} - 3,\quad U_j^{0,-} = -1.2e^{-x_j} + 3,\; j = 1, \ldots,J.
\end{equation*}

\begin{figure}[H]
	\centering
	\begin{subfigure}[b]{0.495\textwidth}
		\centering
		\includegraphics[width=\textwidth]{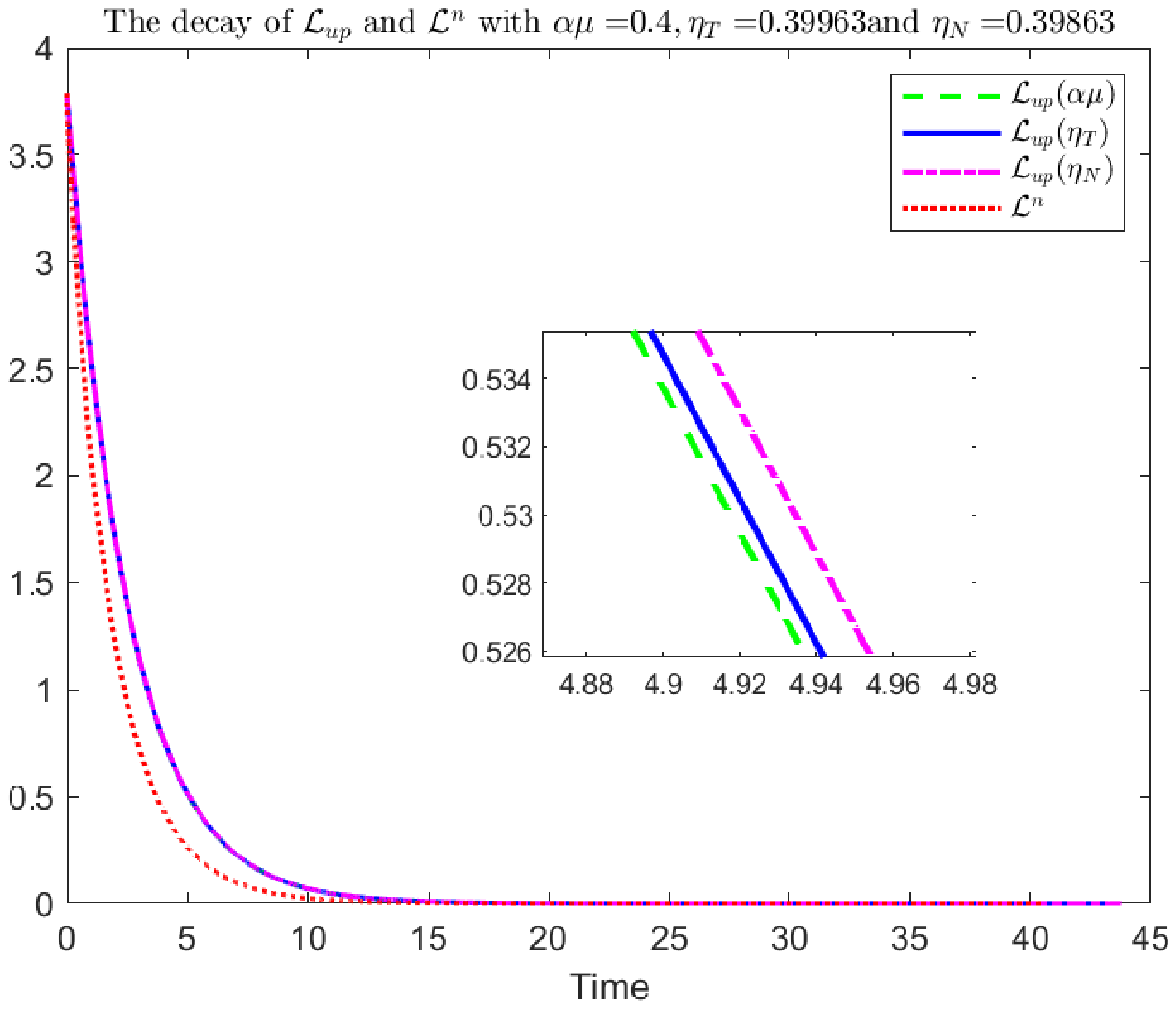}
		%\caption[]{ }    
		\label{fig:DecayRateComparion-03-01}
	\end{subfigure}
	\hfill
	\begin{subfigure}[b]{0.495\textwidth}  
		\centering 
		\includegraphics[width=\textwidth]{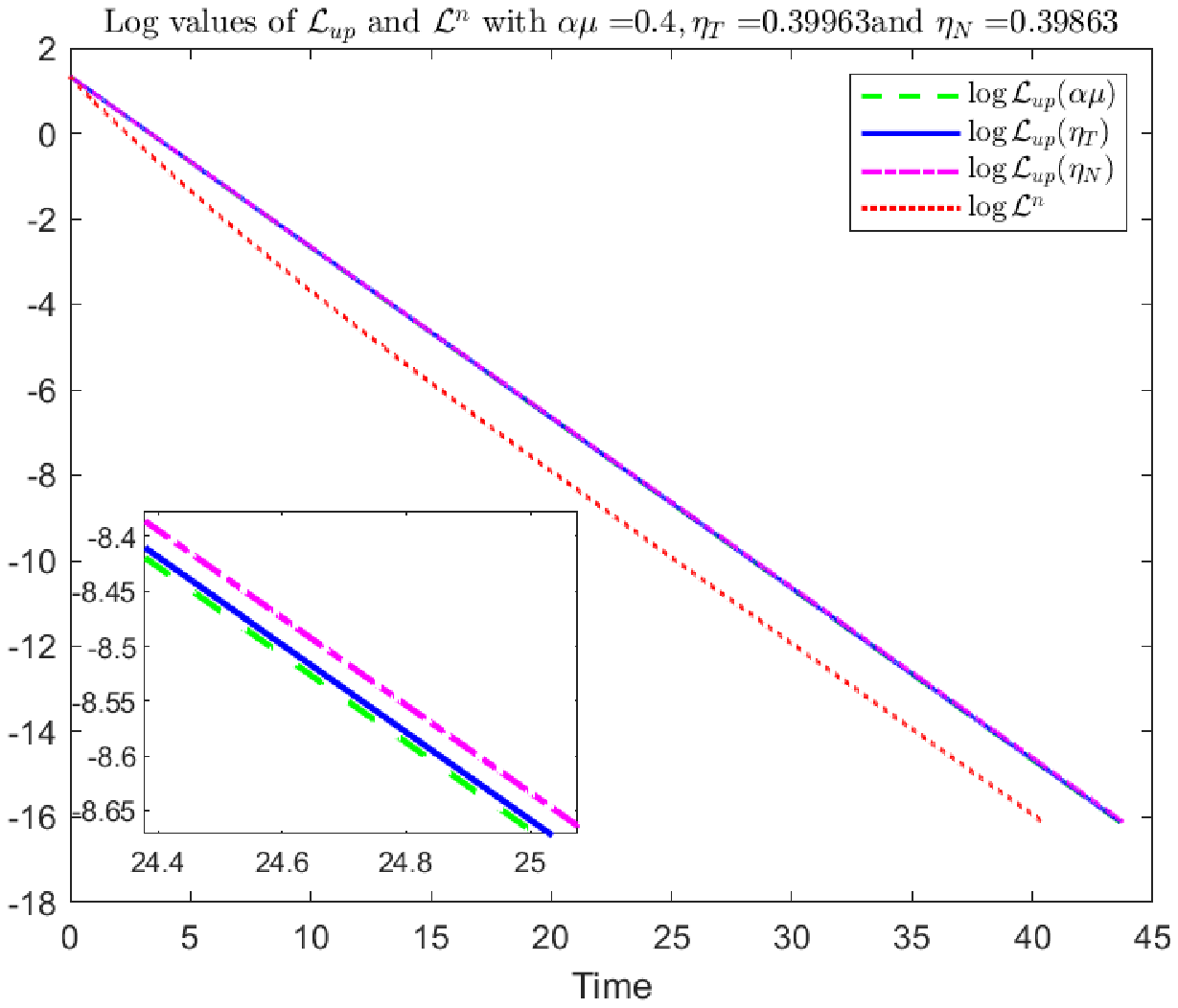}
		%\caption[]{ }    
		\label{fig:DecayRateComparion-03-02}
	\end{subfigure}
	\caption[]%
	{ The comparison of the decay of Lyapunov function and the upper bound in and its log values. Under $ {\Delta x} = \frac{1}{200} $, $ {\Delta t} = \text{CFL}{\Delta x}$, CFL = 0.5, $ \mu = 0.5 $ and $ T = 45 $.} 
	\label{fig:LyapunovProfile-03}
\end{figure}

\begin{figure}[H]
	\centering
	\begin{subfigure}[b]{0.495\textwidth}
		\centering
		\includegraphics[width=\textwidth]{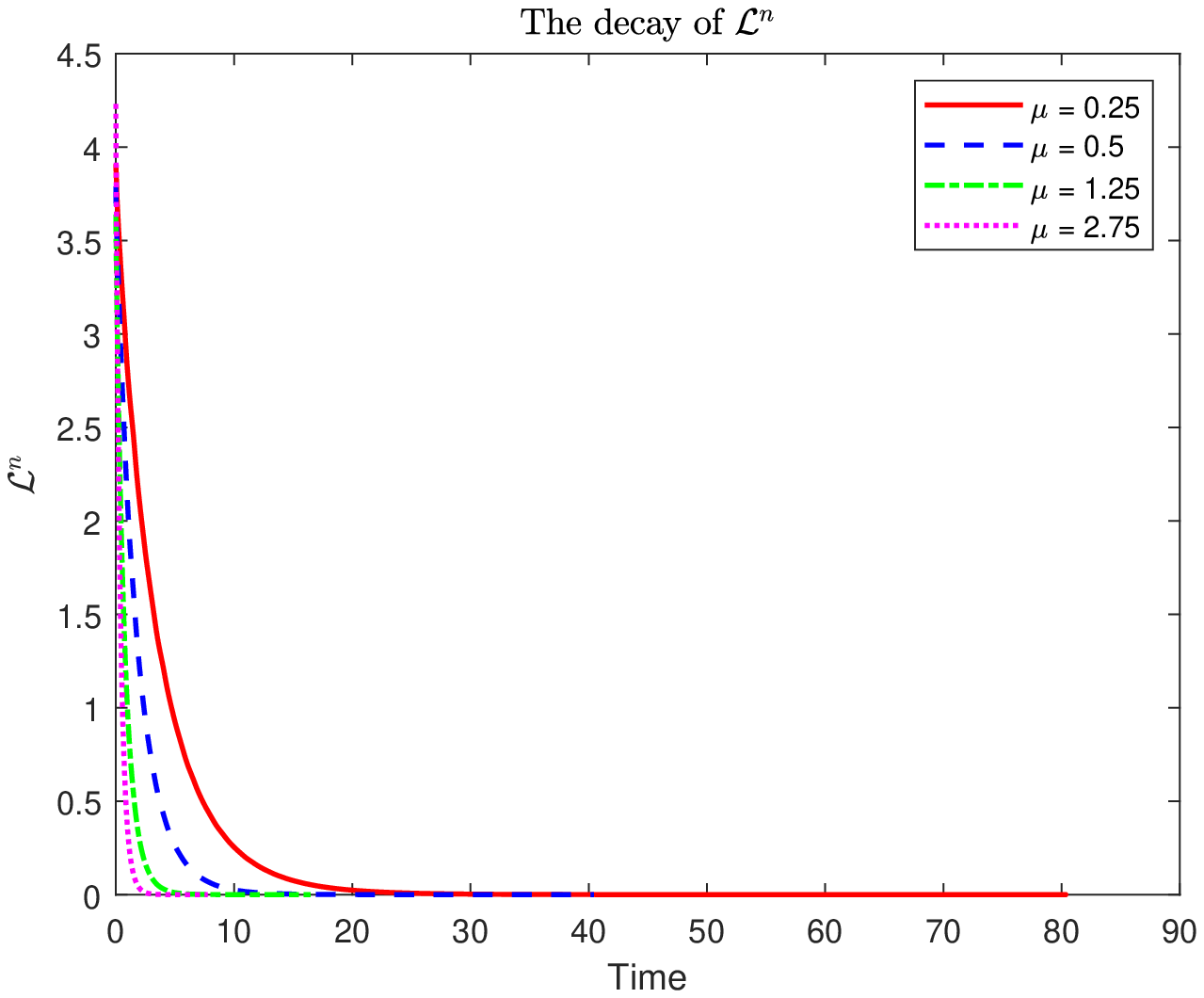}
		%\caption[]{ }    
		\label{fig:DecayRateComparion-04-01}
	\end{subfigure}
	\hfill
	\begin{subfigure}[b]{0.495\textwidth}  
		\centering 
		\includegraphics[width=\textwidth]{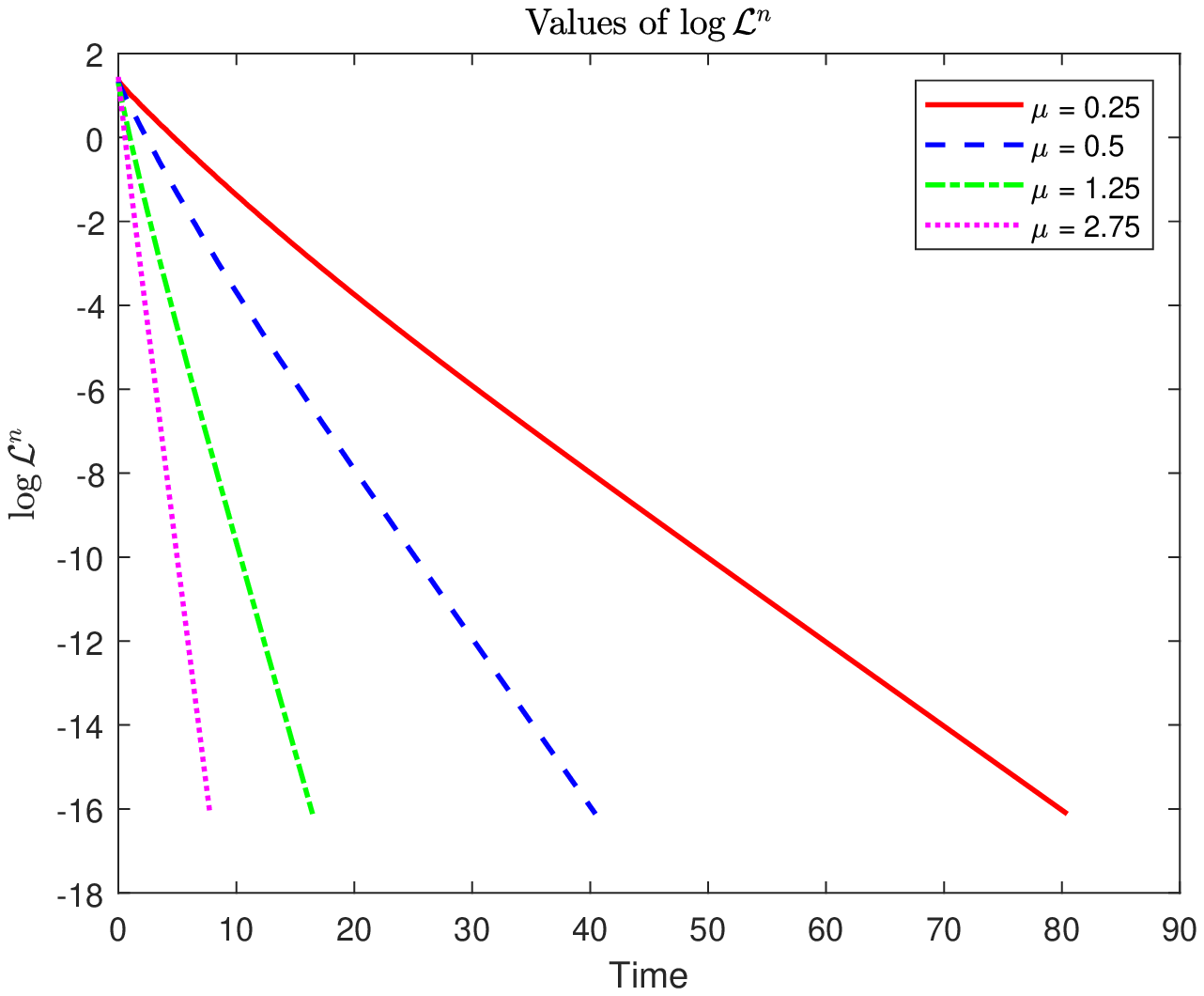}
		%\caption[]{ }    
		\label{fig:DecayRateComparion-04-02}
	\end{subfigure}
	\caption[]%
	{ The comparison of the decay of Lyapunov function and its log values for different choices of $ \mu > 0 $. Under $ {\Delta x} = \frac{1}{200} $, $ {\Delta t} = \text{CFL}{\Delta x}$, CFL = 0.5, and $ T = 90$.} 
	\label{fig:LyapunovProfile-04}
\end{figure}
Similar to Section \ref{sec:linearwave}, in Figure \ref{fig:LyapunovProfile-03}, we compare different upper-bounds of the Lyapunov function with the $\mathcal{L}^n$ which is the discrete Lyapunov function stated in Equation \eqref{eq:dLyafun-system}.  It can also be seen that $\mathcal{L}_{\text{up}}(\alpha\mu) \le \mathcal{L}_{\text{up}}(\eta_T) \le \mathcal{L}_{\text{up}}(\eta_N)$.  For fixed $\mu$, we see that $\mathcal{L}^n$ has the fastest decay rate which can be attributed to the second-order terms in Equation \eqref{eq:discModifiedlinsystem}. We also show that when the values of $ \mu > 0 $ increase the decay rate increases, see Figure \ref{fig:LyapunovProfile-04}. This can also be attributed to the boundary condition in which the matrix $K$ depends on $\mu$.

\subsection{Example 3 - Saint-Venant equations}\label{sec:SaintVenant}
We consider a flow of water along an open channel of prismatic shape of a unit width and length of $ l > 0 $. This flow is modeled by Saint-Venant equations of the form
\begin{equation}\label{eq:ShallowWEqns01}
\begin{split}
\partial_t h(t,x) + \partial_xq(t,x) =&\; 0,\\
\partial_tq(t,x) + \partial_x\left(\frac{q^2(t,x)}{h(t,x)} + \frac{1}{2}g h^2(t,x)\right) =& \; 0,\quad 
\end{split}
\end{equation}
where $ h(t,x) $ is height of water, $ q(t,x):=h(t,x)u(t,x) $ is mass flux, $ u(t,x) $ is the velocity of water and $ g $ is the gravitational constant. The system \eqref{eq:ShallowWEqns01} can be written as the system \eqref{linSystem} with $ a^+ = q^*/h^* + \sqrt{gh^*} $, $ a^- = q^*/h^* - \sqrt{gh^*} $ (a sub-critical flow is considered), and the Riemann invariants can be obtained by \eqref{RiemannCoordinates} as 
\begin{equation*}
U^+(t,x):= (q(t,x) - q^*) - a^-(h(t,x) - h^*), \quad U^-(t,x):= (q(t,x) - q^*) - a^+(h(t,x) - h^*),
\end{equation*}
where $ h^* $, $ q^* $ is a steady-state solution of the system \eqref{eq:ShallowWEqns01} that satisfies 
\begin{equation}\label{eq:ShallowWEqns02}
q^* =\; \kappa_1,\quad \frac{{u^*}^2}{ h^*} + \frac{1}{2}g {h^*}^2 = \; \kappa_2,
\end{equation}
where $ \kappa_1 $ and $ \kappa_2 $ are real constants. Beside this, the discrete system \eqref{eq:discModified} approximates \eqref{eq:ShallowWEqns01}.

We take an equilibrium solution $ q^* = 10 $ and $ h^* = 4 $ with $ g = 9.8 $. Thus, $ a^+ = 8.761 $ and $ a^- = -3.761 $. We define initial data by 
\begin{equation*}
U_j^{0,+} = 10 + 3.761(4 + 0.5\sin(\pi x_j)),\quad U_j^{0,-} = 10 - 8.761(4 + 0.5\sin(\pi x_j)),\; j = 1, \ldots,J.
\end{equation*}

\begin{figure}[H]
	\centering
	\begin{subfigure}[b]{0.495\textwidth}
		\centering
		\includegraphics[width=\textwidth]{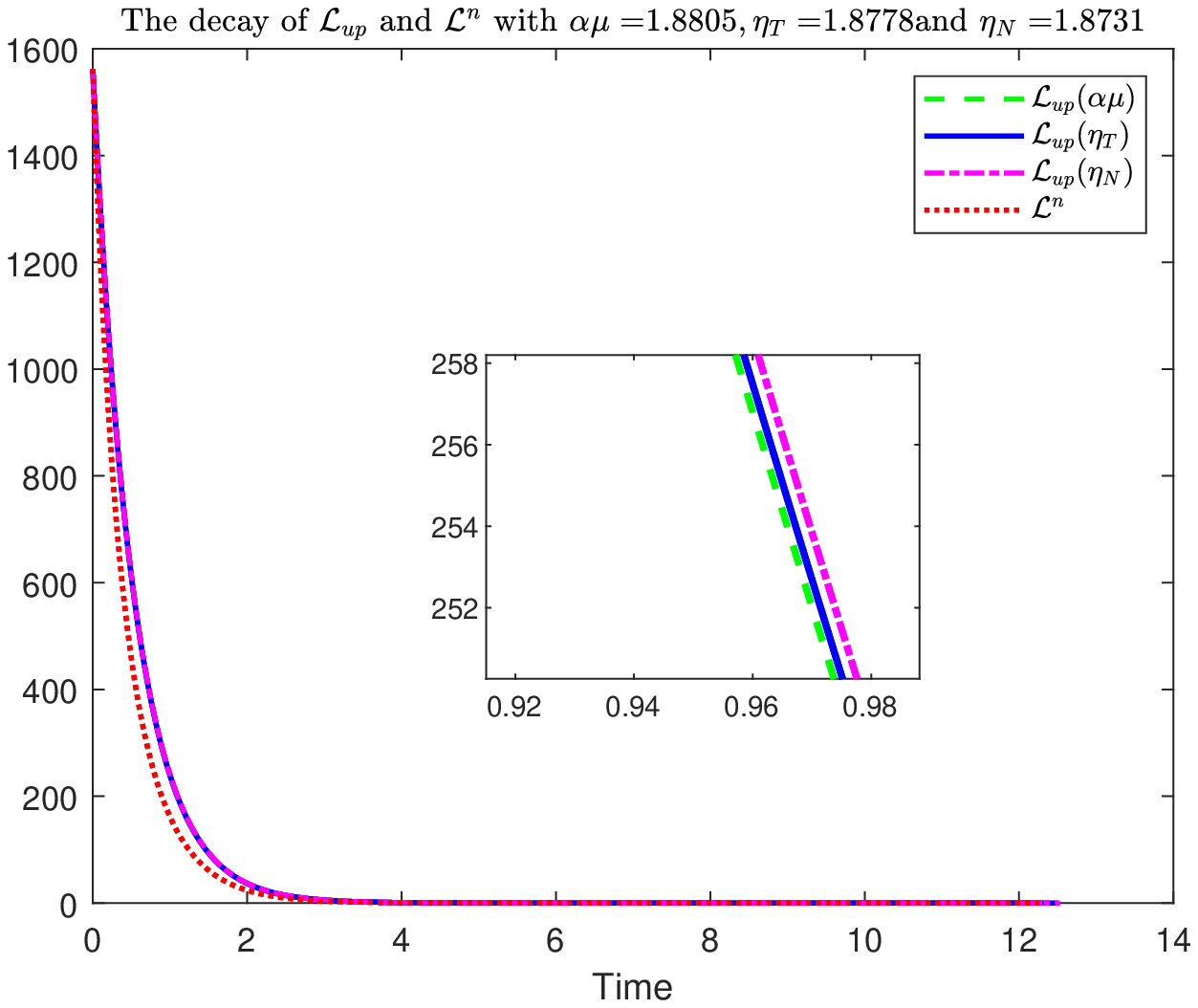}
		%\caption[]{ }    
		\label{fig:DecayRateComparion-05-01}
	\end{subfigure}
	\hfill
	\begin{subfigure}[b]{0.495\textwidth}  
		\centering 
		\includegraphics[width=\textwidth]{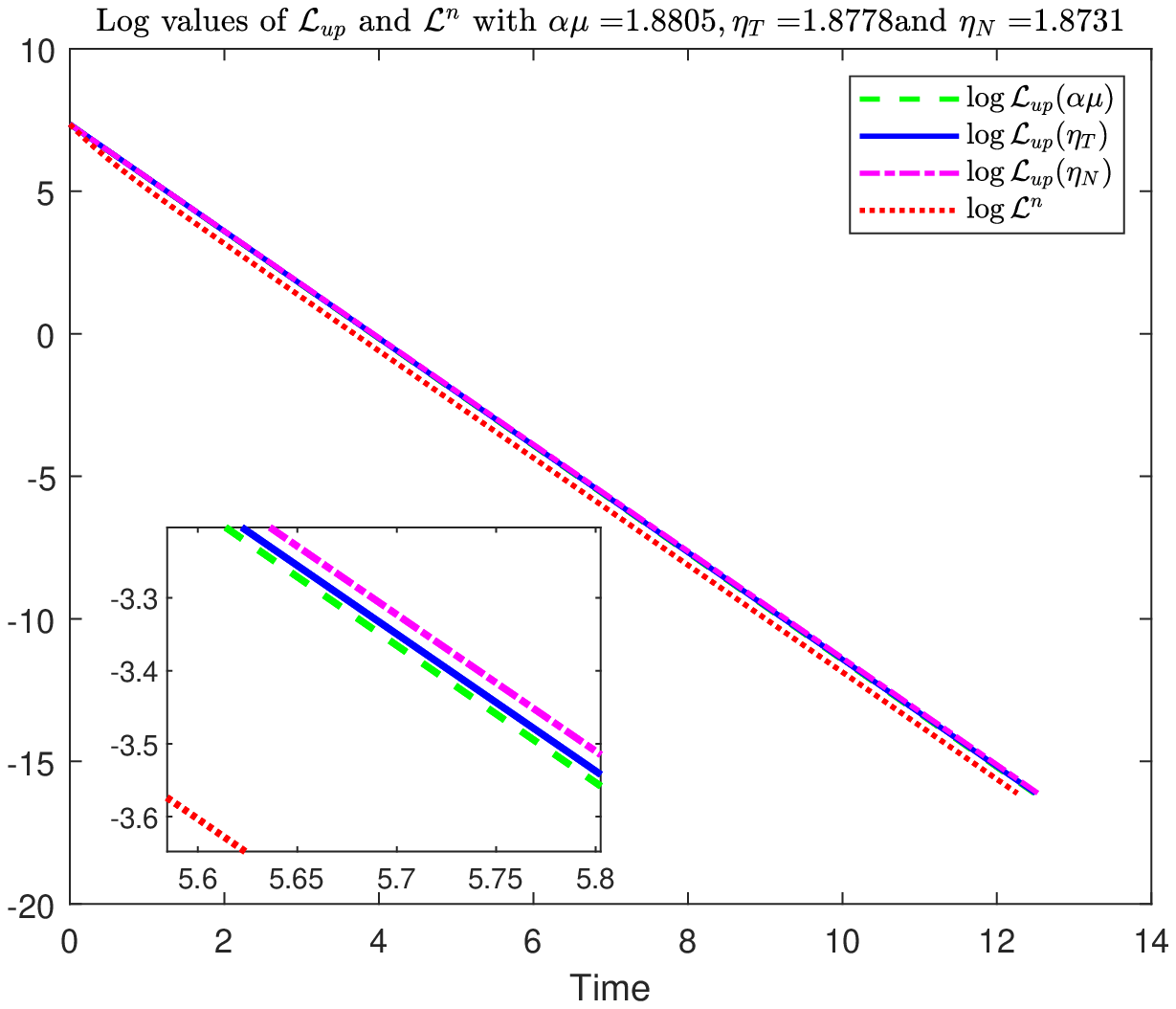}
		%\caption[]{ }    
		\label{fig:DecayRateComparion-05-02}
	\end{subfigure}
	\caption[]%
	{ The comparison of the decay of Lyapunov function and the upper bound in and its log values. Under $ {\Delta x} = \frac{1}{200} $, $ {\Delta t} = \text{CFL}{\Delta x}$, CFL = 0.5, $ \mu = 0.5 $ and $ T = 14 $.} 
	\label{fig:LyapunovProfile-05}
\end{figure}

\begin{figure}[H]
	\centering
	\begin{subfigure}[b]{0.495\textwidth}
		\centering
		\includegraphics[width=\textwidth]{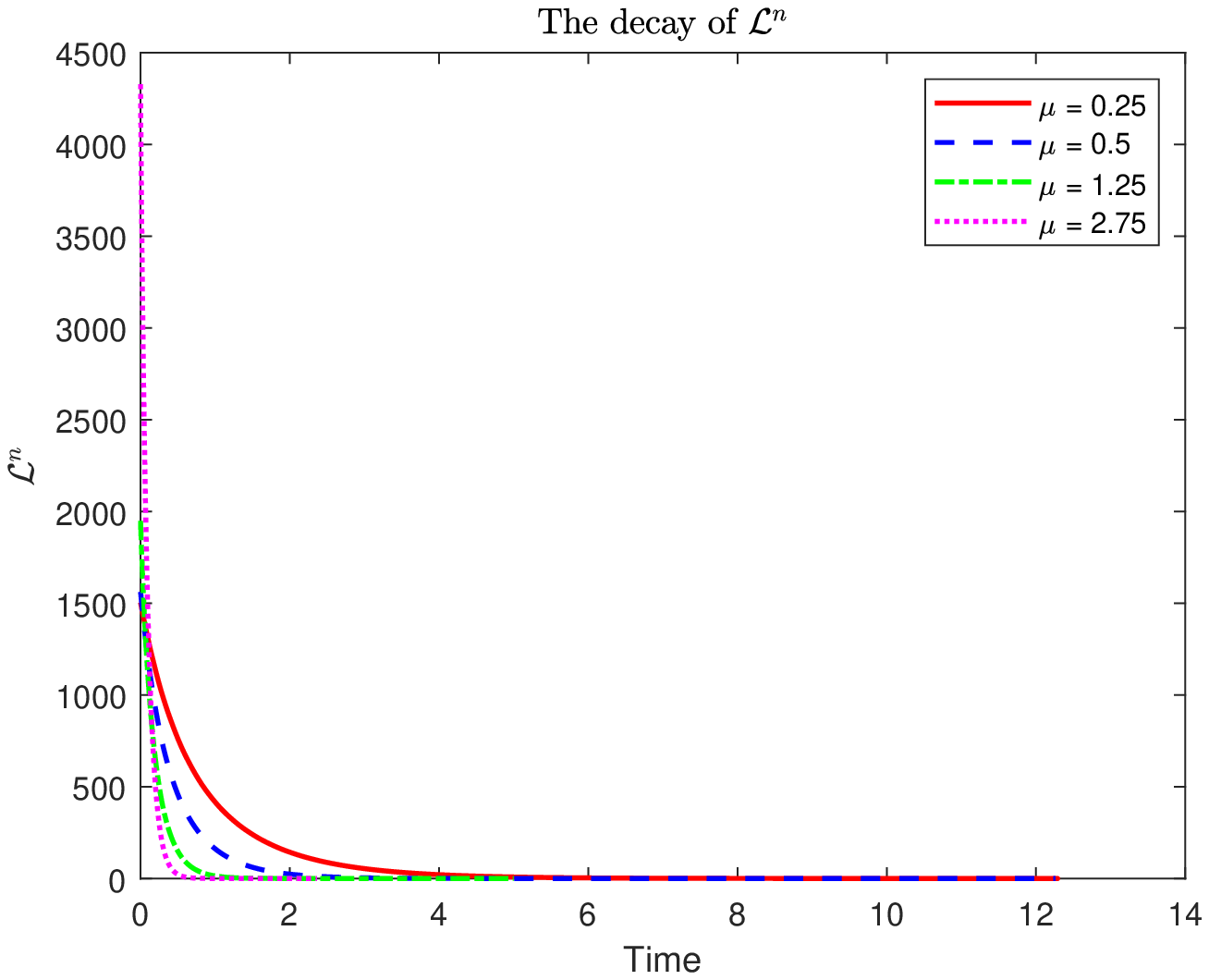}
		%\caption[]{ }    
		\label{fig:DecayRateComparion-06-01}
	\end{subfigure}
	\hfill
	\begin{subfigure}[b]{0.495\textwidth}  
		\centering 
		\includegraphics[width=\textwidth]{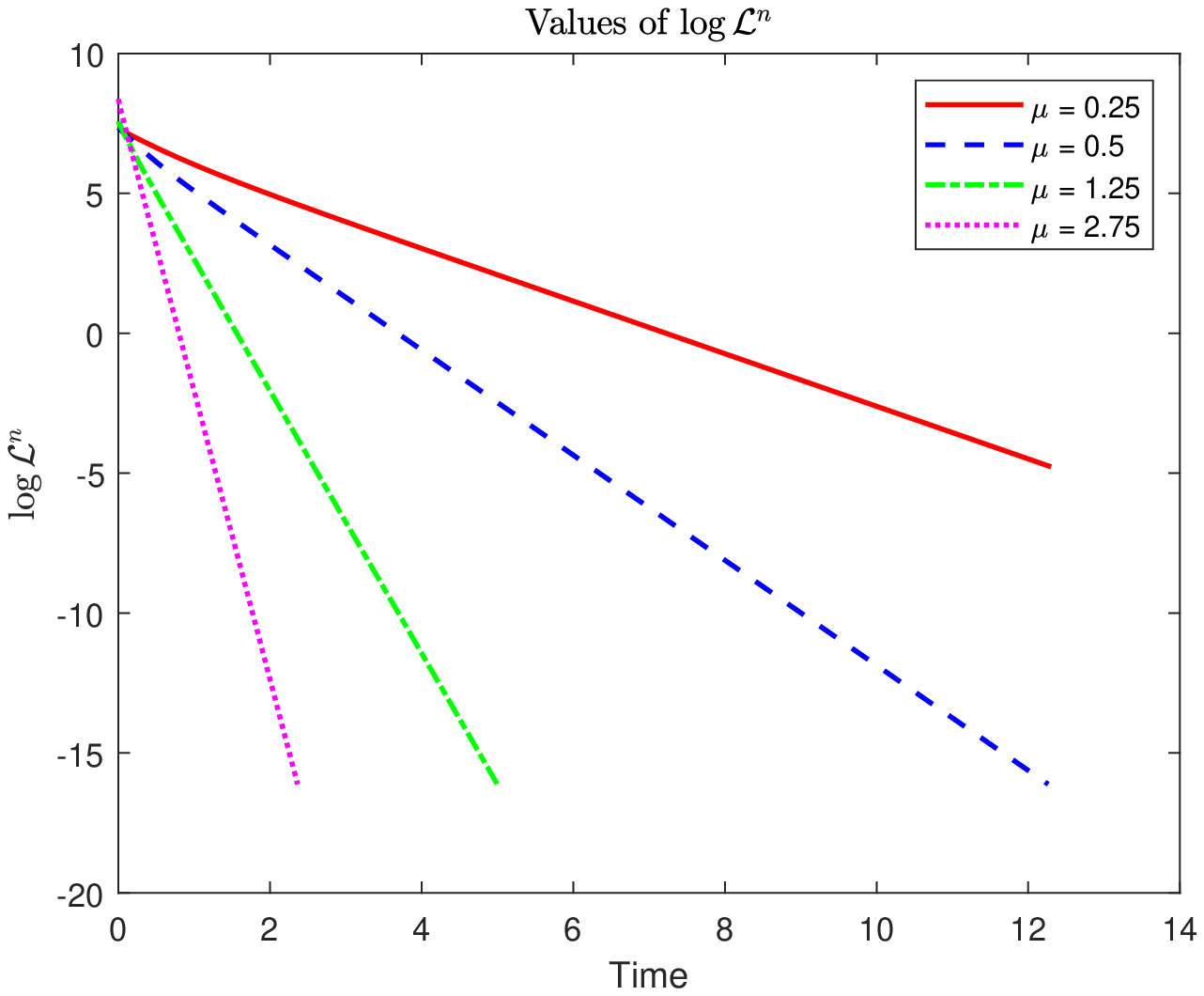}
		%\caption[]{ }    
		\label{fig:DecayRateComparion-06-02}
	\end{subfigure}
	\caption[]%
	{ The comparison of the decay of Lyapunov function and its log values for different choices of $ \mu > 0 $. Under $ {\Delta x} = \frac{1}{200} $, $ {\Delta t} = \text{CFL}{\Delta x}$, CFL = 0.5, and $ T = 14 $.} 
	\label{fig:LyapunovProfile-06}
\end{figure}

Similar to Section \ref{sec:linearwave} and \ref{sec:isothermalEuler}, in Figure \ref{fig:LyapunovProfile-05}, we compare different upper-bounds of the Lyapunov function with the $\mathcal{L}^n$ which is the discrete Lyapunov function stated in Equation \eqref{eq:dLyafun-system}.  It can also be seen that $\mathcal{L}_{\text{up}}(\alpha\mu) \le \mathcal{L}_{\text{up}}(\eta_T) \le \mathcal{L}_{\text{up}}(\eta_N)$.  For fixed $\mu$, we see that $\mathcal{L}^n$ has the fastest decay rate which can be attributed to the second-order terms in Equation \eqref{eq:discModifiedlinsystem}. We also show that when the values of $ \mu > 0 $ increase the decay rate increases, see Figure \ref{fig:LyapunovProfile-06}. This can also be attributed to the boundary condition in which the matrix $K$ depends on $\mu$.

\section{Conclusion}
In this paper a modified equation of linear hyperbolic systems of conservation laws was considered to show the effect of artificial viscosity in numerical boundary feedback control stabilisation. It has been proved that using the upwind scheme and depending on the CFL number that is applied, the decay rates vary. The decay rates when the CFL number is equal to one reflect the expected theoretical decay rates. Otherwise the rates depend on the values of the CFL number and the second derivative terms. This establishes the conjecture that was proposed in \cite{M.K.Banda:2013aa}. As further work, it will be interesting to analyse the effect of numerical schemes on the decay of the Lyapunov function for higher-order discretisation schemes.

\section*{Acknowledgments} 
The authors would like to thank Prof. M. Herty of IGPM, RWTH Aachen University, Germany for constructive and fruitful discussions on the subject addressed in this paper. They would further like to acknowledge support in part by the National Research Foundation of South Africa (Grant Number: 93099 and 102563) and the DFG grant number: GO 1920/10-1. 

%---------------------------
\bibliographystyle{plainnat}
\bibliography{Ref_ArtificialViscosity}

\begin{thebibliography}{27}
\providecommand{\natexlab}[1]{#1}
\providecommand{\url}[1]{\texttt{#1}}
\expandafter\ifx\csname urlstyle\endcsname\relax
  \providecommand{\doi}[1]{doi: #1}\else
  \providecommand{\doi}{doi: \begingroup \urlstyle{rm}\Url}\fi

\bibitem[Banda and Herty(2013)]{M.K.Banda:2013aa}
M.~K. Banda and M.~Herty.
\newblock Numerical discretization of stabilization problems with boundary
  controls for systems of hyperbolic conservation laws.
\newblock \emph{Mathematical control and related fields}, 3\penalty0
  (2):\penalty0 121--142, 2013.
\newblock \doi{10.3934/mcrf.2013.3.121}.

\bibitem[Banda et~al.(2006)Banda, Herty, and Klar]{Banda_2006}
Mapundi Banda, Michael Herty, and Axel Klar.
\newblock Coupling conditions for gas networks governed by the isothermal
  {E}uler equations.
\newblock \emph{Networks and Heterogeneous Media}, 1\penalty0 (2):\penalty0
  295--314, mar 2006.
\newblock \doi{10.3934/nhm.2006.1.295}.
\newblock URL \url{https://doi.org/10.3934%2Fnhm.2006.1.295}.

\bibitem[Banda and Weldegiyorgis(2018)]{Banda_2018}
Mapundi~K. Banda and Gediyon~Y. Weldegiyorgis.
\newblock Numerical boundary feedback stabilisation of non-uniform hyperbolic
  systems of balance laws.
\newblock \emph{International Journal of Control}, pages 1--14, Aug 2018.
\newblock \doi{10.1080/00207179.2018.1509133}.
\newblock URL \url{https://doi.org/10.1080}.

\bibitem[Bastin and Coron(2016)]{Bastin:2016aa}
Georges Bastin and Jean-Michel Coron.
\newblock \emph{Stability and boundary stabilization of 1-D hyperbolic
  systems}.
\newblock Birkhauser, Verlag, Basel, 2016.

\bibitem[Carpentier et~al.(1997)Carpentier, de~La~Bourdonnaye, and
  Larrouturou]{Carpentier_1997}
Romuald Carpentier, Armel de~La~Bourdonnaye, and Bernard Larrouturou.
\newblock On the derivation of the modified equation for the analysis of linear
  numerical methods.
\newblock \emph{{ESAIM}: Mathematical Modelling and Numerical Analysis},
  31\penalty0 (4):\penalty0 459--470, 1997.
\newblock \doi{10.1051/m2an/1997310404591}.
\newblock URL \url{https://doi.org/10.1051%2Fm2an%2F1997310404591}.

\bibitem[Coron(2002)]{J.-M.Coron:2002aa}
J.-M. Coron.
\newblock Local controllability of a 1-{D} tank containing a fluid modeled by
  the shallow water equations.
\newblock \emph{ESAIM: Control, Optimisation and calculus of variations},
  8:\penalty0 513--554, 2002.
\newblock \doi{10.1051/cocv:2002050}.

\bibitem[Coron and Bastin(2015)]{J.-M.Coron:2015aa}
J.-M. Coron and G.~Bastin.
\newblock Dissipative boundary conditions for one-dimensional quasilinear
  hyperbolic systems: {L}yapunov stability for the {$C^1$}-norm.
\newblock \emph{SIAM J.~Control Optim.}, 53\penalty0 (3):\penalty0 1464--1483,
  2015.
\newblock \doi{DOI:10.1137/14097080X}.

\bibitem[Coron et~al.(2007{\natexlab{a}})Coron, Bastin, and
  D'Andrea-Novel]{J.-M.Coron:2007ac}
J.-M. Coron, G.~Bastin, and B.~D'Andrea-Novel.
\newblock A strict {L}yapunov function for boundary control of hyperbolic
  systems of conservation laws.
\newblock \emph{Conference Paper; IEEE Transactions on Automatic Control},
  52\penalty0 (1):\penalty0 2--11, 2007{\natexlab{a}}.
\newblock \doi{10.1109/TAC.2006.887903}.

\bibitem[Coron et~al.(2007{\natexlab{b}})Coron, Bastin, D'Andrea-Novel, and
  Haut]{J.-M.Coron:2007ab}
J.-M. Coron, G.~Bastin, B.~D'Andrea-Novel, and B.~Haut.
\newblock {L}yapunov stability analysis of networks of scalar conservation
  laws.
\newblock \emph{Networks and heterogeneous media}, 2\penalty0 (4):\penalty0
  749--757, 2007{\natexlab{b}}.

\bibitem[Coron et~al.(2008{\natexlab{a}})Coron, Bastin, and
  D'Andrea-Novel]{J.-M.Coron:2008aa}
J.-M. Coron, G.~Bastin, and B.~D'Andrea-Novel.
\newblock Dissipative boundary conditions for one dimensional nonlinear
  hyperbolic systems.
\newblock \emph{SIAM J.~Control Optim.}, 47\penalty0 (3):\penalty0 1460--1498,
  2008{\natexlab{a}}.
\newblock \doi{10.1137/070706847}.

\bibitem[Coron et~al.(2008{\natexlab{b}})Coron, Bastin, and
  D'Andrea-Novel]{J.-M.Coron:2008ab}
J.-M. Coron, G.~Bastin, and B.~D'Andrea-Novel.
\newblock Boundary feedback control and {L}yapunov stability analysis for
  physical networks of 2$\times$2 hyperbolic balance laws.
\newblock \emph{Proceedings of the 47th IEEE Conference on decision and
  control}, 2008{\natexlab{b}}.

\bibitem[Coron et~al.(2009)Coron, Bastin, and
  D'Andrea-Novel]{J.-M.Coron:2009aa}
J.-M. Coron, G.~Bastin, and B.~D'Andrea-Novel.
\newblock On {L}yapunov stability of linearised {S}aint-{V}enant equations for
  a sloping channel.
\newblock \emph{Networks and heterogeneous media}, 4:\penalty0 177--187, 2009.
\newblock \doi{10.3934/nhm.2009.4.177}.

\bibitem[Coron(2009)]{Coron_2009}
Jean-Michel Coron.
\newblock \emph{Control and Nonlinearity}.
\newblock American Mathematical Society, Aug 2009.
\newblock \doi{10.1090/surv/136}.
\newblock URL \url{https://doi.org/10.1090%2Fsurv%2F136}.

\bibitem[Dafermos(2010)]{C.M.Dafermos:2010aa}
C.~M. Dafermos.
\newblock \emph{Hyperbolic conservation laws in continuum physics}.
\newblock A series of comprehensive studies in mathematics. Springer,
  Providence, RI, 2010.
\newblock 3rd ed.

\bibitem[de~Halleux et~al.(2003)de~Halleux, Prieur, Coron, D'Andrea-Novel, and
  Bastin]{J.deHalleux:2003aa}
J.~de~Halleux, C.~Prieur, J.-M. Coron, B.~D'Andrea-Novel, and G.~Bastin.
\newblock Boundary feedback control in networks of open channels.
\newblock \emph{Automatica}, 39:\penalty0 1365--1376, 2003.

\bibitem[Dick et~al.(2010)Dick, Gugat, and Leugering]{M.Dick:2010aa}
M.~Dick, M.~Gugat, and G.~Leugering.
\newblock Classical solutions and feedback stabilization for the gas flow in a
  sequence of pipes.
\newblock \emph{Networks and heterogeneous media}, 5\penalty0 (4):\penalty0
  691--709, 2010.
\newblock \doi{10.3934/nhm.2010.5.691}.

\bibitem[Gerster and Gugat(2019)]{Gerster:2019aa}
Stephan Gerster and Martin Gugat.
\newblock On the limits of stabilizability for networks of strings.
\newblock \emph{Systems \& Control Letters}, 131, 2019.

\bibitem[Gugat and Herty(2011)]{M.Gugat:2011aa}
M.~Gugat and M.~Herty.
\newblock Existence of classical solutions and feedback stabilization for the
  flow in gas networks.
\newblock \emph{ESAIM: Control, optimisation and calculus of variations},
  17:\penalty0 28--51, 2011.
\newblock \doi{10.1051/cocv/2009035}.

\bibitem[Gugat and Leugering(2003)]{M.Gugat:2003aa}
M.~Gugat and G.~Leugering.
\newblock Global boundary controllability of the de {S}t.~{V}enant equations
  between steady states.
\newblock \emph{Ann.~Inst.~Henri Poincar\'{e} Anal.~Non Lin\'{e}aire},
  20\penalty0 (1):\penalty0 1--11, 2003.

\bibitem[Gugat et~al.(2004)Gugat, Leugering, and Schmidt]{M.Gugat:2004aa}
M.~Gugat, G.~Leugering, and G.~Schmidt.
\newblock Global controllability between steady supercritical flows in channel
  networks.
\newblock \emph{Mathematical methods in the applied science}, 27:\penalty0
  781--802, 2004.
\newblock \doi{10.1002/mma}.

\bibitem[Gugat et~al.(2012)Gugat, Leugering, Tamasoiu, and
  Wang]{M.Gugat:2012aa}
M.~Gugat, G.~Leugering, S.~Tamasoiu, and K.~Wang.
\newblock {$H^2$}-stabilization of the isothermal {E}uler equations: a
  {L}yapunov function approach.
\newblock \emph{Chin.~Ann.~Math.}, 4:\penalty0 479--500, 2012.
\newblock \doi{10.1007/s11401-012-0727-y}.

\bibitem[Herty and Gerster(2019)]{Herty:2019aa}
Michael Herty and Stephan Gerster.
\newblock Discretized feedback control for systems of linearized hyperbolic
  balance laws.
\newblock \emph{Math. Control Relat. Fields}, 9\penalty0 (3), 2019.

\bibitem[Komornik and Gattulli(1997)]{Komornik:1997aa}
Vilmos Komornik and Vincenzo Gattulli.
\newblock Exact controllability and stabilization. the multiplier method.
\newblock \emph{SIAM Review}, 39\penalty0 (2):\penalty0 351--351, 1997.

\bibitem[Leugering and Schmidt(2002)]{G.Leugering:2002aa}
G.~Leugering and G.~Schmidt.
\newblock On the modelling and stabilization of flows in networks of open
  canals.
\newblock \emph{SIAM J.~Control Optim.}, 41:\penalty0 164--180, 2002.

\bibitem[LeVeque(2002)]{LeVeque:2002aa}
Randall~J LeVeque.
\newblock \emph{Finite volume methods for hyperbolic problems}, volume~31.
\newblock Cambridge university press, 2002.

\bibitem[Schillen and G\"ottlich(2017)]{P.Schillen:2017aa}
P.~Schillen and S.~G\"ottlich.
\newblock Numerical discretization of boundary control problems for systems of
  balance laws: Feedback stabilization.
\newblock \emph{European Journal of Control}, 35:\penalty0 11--18, 2017.

\bibitem[Weldegiyorgis and Banda(2019)]{Weldegiyorgis_2019}
Gediyon~Y. Weldegiyorgis and Mapundi~K. Banda.
\newblock An analysis of the input-to-state stabilisation of linear hyperbolic
  systems of balance laws with additive disturbances.
\newblock \emph{preprint}, 2019.

\end{thebibliography}

\end{document}